\documentclass[10pt]{amsart}

\usepackage{amsthm,amsmath,amsfonts,graphicx}

\newtheorem{Theorem}{Theorem}[section]
\newtheorem{Lemma}{Lemma}[section]
\newtheorem{Proposition}{Proposition}[section]

\theoremstyle{remark}
\newtheorem{Remark}{Remark}[section]

\theoremstyle{definition}

\renewcommand{\P}{\ensuremath{\mathbb{P}\,}}
\newcommand{\E}{\ensuremath{\mathbb{E}\,}}

\newcommand{\N}{\ensuremath{\mathbb{N}}}

\renewcommand{\d}{\ensuremath{\,{\rm d}}}
\newcommand{\e}{\ensuremath{{\rm e}}}
\newcommand{\var}{\ensuremath{{\rm var}\,}}

\begin{document}


\title[Pathwise LSE for SPDEs]{Pathwise least-squares estimator for linear SPDEs with additive fractional noise}


\author{Pavel K\v r\'i\v z}
\address{Charles University, Faculty of Mathematics and Physics, Sokolovsk\' a 83, Prague 8, 186 75, Czech Republic.
\newline
University of Chemistry and Technology,  Prague, Department of Mathematics, Technick\' a 5, Prague 6, Czech Republic. 
}
\email{kriz@karlin.mff.cuni.cz}

\author{Jana \v Snup\' arkov\' a}
\address{University of Chemistry and Technology,  Prague, Department of Mathematics, Technick\' a 5, Prague 6, Czech Republic. }
\email{jana.snuparkova@vscht.cz}

\thanks{The work of PK was supported by the Czech Science Foundation project No. 19-07140S and the work of J\v S was supported by the grant LTAIN19007 Development of Advanced Computational Algorithms for Evaluating Post-surgery~Rehabilitation.}








\begin{abstract}
This paper deals with the drift estimation in linear stochastic evolution equations (with emphasis on linear SPDEs) with additive fractional noise (with Hurst index ranging from 0 to 1) via least-squares procedure. Since the least-squares estimator contains stochastic integrals of divergence type, we address the problem of its pathwise (and robust to observation errors) evaluation by comparison with the pathwise integral of Stratonovich type and using its chain-rule property. The resulting pathwise LSE is then defined implicitly as a solution to a non-linear equation. We study its numerical properties (existence and uniqueness of the solution) as well as statistical properties (strong consistency and the speed of its convergence). The asymptotic properties are obtained assuming fixed time horizon and increasing number of the observed Fourier modes (space asymptotics). We also conjecture the asymptotic normality of the pathwise LSE.
\end{abstract}

\maketitle

\section{Introduction}
Least-squares type estimators of an unknown drift parameter have recently become very popular in the setting of (semi)linear SPDEs driven by an infinite-di\-mensional Brownian motion, because they have advantageous asymptotic properties (they basically coincide with MLE estimators in this case) and are relatively easy to implement. For recent theoretical works, the reader may check \cite{PS20}, \cite{CDK20}, \cite{AR20_pre} or \cite{ACP20_pre} to name just a few. The theory for models with Brownian motion has reached certain level of maturity so that it enables several interesting practical applications, such as \cite{ABJR20_pre} or \cite{P_et_al21}.

Wide range of randomly evolving phenomena are forced with auto-correlated noise, which can be effectively modeled by a fractional Brownian motion (fBm). Stochastic evolution equations driven by a fBm appear in diverse fields such as biology, neuroscience, hydrology, climatology, finance and many others (see e.g. monographs \cite{Biagini_et_al_2008} or \cite{Mishura_2008} for discussion). Least-squares estimators of the drift parameter for continuously observed trajectories of solutions to one-dimensional linear equations with additive fractional noise have been studied in \cite{HN10} for Hurst index $H\geq 1/2$ and \cite{HNZ17} for general $H \in (0,1)$. 

The available literature on parameter estimation for infinite-dimensional equations driven by an (infinite-dimensional) fBm is rather sparse. A least-squares estimator for a one-dimensional projection of the solution has been studied in long-span asymptotic regime in \cite{MT13} assuming the regular case $H \geq 1/2$. Drift estimation in spectral setting (Fourier modes are observed) for infinite-dimensional equations with fractional noise has been studied in \cite{CLP09} (MLE approach, $H \geq 1/2$) and \cite{Kriz2020} (ergodic-type estimator, $H \in (0,1)$).  We are not aware of any study on LSE in spectral setting (Fourier modes are observed) for these models. 

LSE-type estimators typically contain stochastic integrals, which makes them difficult to apply. For equations with Brownian noise, evaluations of It\^{o} stochastic integrals may be sensitive to small perturbations of trajectories. A robust version of LSE can be obtained by application of It\^{o} formula (for simple equations) or rough path theory (more complex equations), see \cite{FrizDiehlMai2016}. The least-squares estimators for equations with fBm incorporate divergence-type integrals (sometimes called the Skorokhod integrals). These are not defined pathwise, but rather as the adjoint operators to Malliavin derivatives, which makes them extremely difficult (if not impossible) to evaluate from the observed trajectory in practice or in simulations, which prevents them from wider use. To avoid these complications, the authors in \cite{HNZ17} switched to the ergodic-type estimator for one-dimensional fractional Ornstein-Uhlenbeck process, which contains only Riemann integrals. Estimators of ergodic type for more general equations driven by a finite-dimensional fBm were recently studied in \cite{PTV20}. Application of this ergodic approach in the spectral setting for infinite-dimensional equations corresponds to the weighted minimum-contrast estimator (weighted MCE) introduced in \cite{Kriz2020}. Such estimator is plausible for stationary processes, but fails for processes at far-from-stationary state and is difficult to generalize to more complex equations, because (in contrast to LSE) it requires ergodicity and precise knowledge of the limiting covariance operator. 

In this paper, we start with derivation of a new spectral version of the least-squares estimator (theoretical LSE) for an unknown drift parameter in linear SPDEs driven by an additive infinite-dimensional fractional Brownian motion and prove its strong consistency in space (incl. speed of convergence) by the tools of Malliavin calculus. We choose to work in spectral setting (observations in spectral space domain rather than in physical space domain are available) and to study asymptotic properties under increasing number of the observed Fourier modes while time horizon is fixed. Such approach makes it possible to consistently (with error approaching zero) estimate the unknown parameter based on observations in fixed time horizon, but with sufficient detail in space (number of Fourier modes). This is particularly useful for observations having limited time horizon but high space resolution/detail. It contrasts to more typical time asymptotics, which assumes increasing time horizon and which is not studied here. Moreover, sever complications with its numerical evaluation has led us to construct a pathwise version of the LSE, which is defined (and computable) from the observed trajectory and which is robust to small perturbations in the observations. 

Our construction of the robust pathwise LSE is inspired by the work \cite{FrizDiehlMai2016}. It is based on calculation of pathwise Stratonovich integral, which can be expressed explicitly in our case, and compensation for difference between Stratonovich and divergence-type integrals. In our setting with fBm, this leads to an implicitly defined estimator. We prove its existence and uniqueness, as well as its strong consistency in space. Surprisingly, if  $H>1/2$, it has slightly better performance than the theoretical LSE. To our best knowledge, this is the first attempt to define a pathwise robust LSE for equations driven by fBm. We believe that this pathwise least-squares approach is applicable also for different models with fBm (such as those in \cite{MT13} or \cite{HNZ17}) and may offer estimators applicable in practice or in simulations. 

Main purpose of this paper is to make the first step in the development of a theoretical framework necessary for studying spectral asymptotic properties as well as performing effective enumeration of the least-squares estimators for stochastic evolution equations with infinite-dimensional auto-correlated noise. The assumed linearity of the operators ensures reasonable clarity and simplicity of the exposition. The next step would be to study this problem in the setting of semilinear equations with non-linear lower order perturbation with numerous practical applications. This, however, requires fine analysis of space regularity of solutions and is beyond the scope of this work.

The paper is organized as follows. In section \ref{sect:Problem setup} we introduce the studied problem in more detail. The two new estimators -- theoretical and pathwise LSE -- are derived in section \ref{sect:Derivation of the estimator}. Existence and uniqueness of the implicitly defined pathwise LSE is studied in section \ref{sect: Existence and uniqueness of pathwise LSE}. Section \ref{sect: Strong consistency} is devoted to the strong consistency of the estimators. A simulation study is presented is section \ref{sect: Simulation study}. Section \ref{sect:Conclusions} summarizes main results and findings of this paper. In Appendix, we briefly describe some elements of Malliavin calculus, which are useful for the asymptotic analysis and outline the connection between cylindrical fractional Brownian motion and Gaussian noise that is white in space and fractional in time.


\section{Problem setup}\label{sect:Problem setup}
\subsection{The model}
Consider the evolution equation
\begin{equation}\label{sEE}
\d X(t) = \lambda A X(t) \d t + B X(t) \d t + \d B^H(t),\quad X(0)=X_0,
\end{equation}
in a real separable Hilbert space~$V$ where $A,B$ are densely defined linear operators, $B^H$ is a~cylindrical fractional Brownian motion on~$V$, $X_0 \in V$ and $\lambda$ is a~positive unknown parameter, which we want to estimate. 

Throughout this paper we assume that the equation~\eqref{sEE} is diagonalizable, i.e. there exists an orthonormal basis $\{e_k\}_{k=1}^{\infty}$ in~$V$ such that for any $k\in\N$ we have
\[A e_k=-\alpha_k e_k, \ B e_k = -\beta_k e_k,\ \langle B^H(t),e_k\rangle_V = B^H_k(t),\]
where $\alpha_k \geq 0$ and $\beta_k \geq 0$. Furthermore, set 
\[\mu_k =\lambda \alpha_k + \beta_k,\ k\in\N.\]

\begin{Remark}
For existence of the solution to \eqref{sEE} in an appropriate interpolation space, constructed via spectral decomposition, it suffices that for some $\gamma \in \mathbb{R}$ 
\begin{equation}\label{eq: existence cond}
    \sum_{k=1}^{\infty}\frac{1}{(1+\mu_k)^{\gamma}} < \infty
\end{equation}
holds. For more details and proof, see e.g. Theorem 2.1. (and Remark 2.1.) in \cite{Kriz2020}. 
\end{Remark}
We assume (throughout the paper) that the existence condition \eqref{eq: existence cond} is satisfied. It also implies that
\begin{equation}\label{lim mu_k}
\lim_{k \to \infty} \mu_k = \infty.
\end{equation} 

Denote the Fourier modes (projections to eigenvectors) of the solution by $x_k(t) = \langle X(t),e_k \rangle_{V}$. By diagonality of the equation \eqref{sEE} these projections satisfy the system of independent one-dimensional equations
\begin{equation}\label{SDE1}
\begin{array}{rcl}
{\rm d} x_k(t)&=&-\mu_k x_k(t)\d t+ \d B^H_k(t),\ t>0,\ k\in\N,\\
x_k(0)&=&\langle X_0,e_k \rangle_{V},
\end{array}
\end{equation}
where $\{B^H_k(t),t\geq 0\}$ are mutually independent real-valued fractional Brownian motions with the same Hurst index $H\in (0,1)$.

The solutions to the equations~\eqref{SDE1} are mutually independent real-valued fractional Orn\-stein-Uhlenbeck processes given by a formula
\begin{equation}\label{OU}
x_k(t)=\e^{-\mu_k t}x_k(0) + \int_0^t \e^{-\mu_k (t-s)} \d B^H_k(t),\ t\geq 0.
\end{equation}
for any $k\in\N$.

\begin{Remark}\label{rem: diff oper}
If the operators $A,B$ from~\eqref{sEE} are negative definite self-adjoint elliptic linear differenetial operators of even orders $2m_1, 2m_2\, (m_1,m_2 \in \N)$, respectively, acting on a compact $d$-dimensional manifold, then (see \cite{SafarovVassiliev1997})
\begin{equation}\label{SPDE_ev}
\alpha_k \sim k^{2m_1/d}, \hbox{  and  } \beta_k \sim k^{2m_2/d},
\end{equation}
where 
\begin{equation}\label{def: sim}
a_k\sim b_k \text{ means that } a_k / b_k \to C \in (0, \infty) \text{ as } k \to \infty.
\end{equation}
Such operators are of great interest as they frequently appear in various stochastic PDEs. 
\end{Remark}

\subsection{Statistical problem}
To estimate the value of $\lambda$, we have continuous sample trajectories of the first $N$ Fourier modes observed on a fixed time-window $[0,T]$ at our disposal, i.e. our data are  $\big\{x_k(t), t\in [0,T]\big\},\, k = 1,\ldots,N$. We aim at deriving an estimator that would be computable from this data and that would be consistent with increasing number of Fourier modes $N \to \infty$ (the so-called space asymptotics), the time horizon $T$ being fixed.

\begin{Remark}
In this paper we assume that the Hurst parameter $H$ is known. If not, it can be consistently determined from a single continuous trajectory of the process $\big\{x_1(t), t \in [0,T]\big\}$ by some of various techniques developed for estimating $H$ for real-valued processes sampled under in-fill asymptotics regime (discrete time observation with fixed time horizon and decreasing time step), cf. \cite{Istas-Lang1997}, \cite{gloter2007}, \cite{Coeurjolly2008}, or \cite{Rosenbaum2008}) to name just a few. If, moreover, the noise processes for individual Fourier modes $B^{H}_k$ have different intensities (volatilities, denote by $\sigma_k$), we can simultaneously and consistently estimate $\sigma_k$ and $H$  from the observed modes under in-fill asymptotics regime using the powers of the second order variations (see \cite{Berzin_et_al2014}, Chapter 3.3) and subsequently eliminate different noise intensities by appropriate rescaling of the observed processes $x_k$. 
\end{Remark}


\section{Derivation of the estimator}\label{sect:Derivation of the estimator}
\subsection {Theoretical  LSE}
We derive the estimator for the drift parameter by applying the least-squares concept to the SDEs understood as a system of linear regression models. In particular, we find~$\lambda$ which minimizes (for fixed time horizont~$T$) the formal sum
\[\sum_{k=1}^N \int_0^T \big(\dot x_k(t) + \mu_k x_k(t)\big)^2\d t.\]
Note that this is only a heuristic technique, since time derivatives $\dot x_k(t)$ do not exist (in the classical sense). However, we can rewrite the minimizer back in terms of the well-defined stochastic integrals and Riemann integrals. It reads
\begin{equation}\label{LSE_theor1}
\hat \lambda_N = -\dfrac{\sum_{k=1}^N \int_0^T \alpha_k x_k(t)\d x_k(t)+\sum_{k=1}^N \int_0^T \alpha_k\beta_k x_k^2(t)\d t}{\sum_{k=1}^N \int_0^T \alpha_k^2 x_k^2(t)\d t},
\end{equation}
where the stochastic integral is understood in the Skorokhod sense (sometimes referred to as the divergence type integral). Applying the stochastic differential from the equation~\eqref{SDE1} we obtain
\begin{equation}\label{LSE_theor2}
\hat \lambda_N = \lambda -\dfrac{\sum_{k=1}^N \alpha_k\int_0^T x_k(t) \d B^H_k(t)}{\sum_{k=1}^N \alpha_k^2\int_0^T x_k^2(t)\d t}.
\end{equation}
This form is convenient for proving the asymptotic properties of the estimator~$\hat \lambda_N$ whereas the form~\eqref{LSE_theor1} is better for derivation of the pathwise version of the LSE for $\lambda$. 

\begin{Remark}
We did not specify the type of stochastic integral with respect to the fBm in the equations \eqref{sEE} or \eqref{SDE1}, because all integration concepts coincide in case of additive noise. However, the choice of the appropriate type of stochastic integration in \eqref{LSE_theor2} (and in \eqref{LSE_theor1}) is critical, since the integrands are stochastic processes here. The choice of the Skorokhod type integral is motivated by the fact that its expectation is zero, which is advantageous for the error term in  \eqref{LSE_theor2} and will be utilized in the sequel. If we choose the pathwise Stratonovich integral, for example, its nonzero expectation would generate unwanted bias in \eqref{LSE_theor2} and the corresponding estimator would not be consistent.
\end{Remark}

\subsection {Robust pathwise LSE}
Although the Skorokhod type integral in \eqref{LSE_theor1} has advantageous probabilistic properties, it is not constructed pathwise, but rather as the adjoint operator to Malliavin derivative. This imposes sever complications to numerical evaluation of the estimator $\hat \lambda_N$ for the observed single trajectory. To overcome this issue, we derive below its robust pathwise version. 

The strategy, inspired by the paper \cite{FrizDiehlMai2016}, is to express the estimator in terms of the pathwise Stratonovich integral (for definition, see \cite{HNZ17}) with appropriate compensation for difference between Stratonovich and Skorokhod integrals.
\begin{equation}
\begin{aligned}
    \int_{0}^{T} x_k(t) \d x_k(t) &=  \int_{0}^{T} - \mu_k x_k^2(t) \d t + \int_{0}^{T} x_k(t) \d B_k^H(t) \\\nonumber
    &=  \int_{0}^{T} - \mu_k x_k^2(t) \d t + \left(\int_{0}^{T} x_k(t) \circ \d B_k^H(t)  -   \Delta(\mu_k) \right)\\
    &=  \int_{0}^{T} x_k(t) \circ \d x_k(t) -  \Delta(\mu_k),
\end{aligned}
\end{equation}
where $\int_{0}^{T} x_k(t) \circ \d x_k(t)$ stands for the pathwise Stratonovich integral and $\Delta(\mu_k)$ is the compensation. Since $x_k$ is in the first Wiener chaos (it is Gaussian), the compensation $\Delta(\mu_k)$ takes rather simple form (see \cite{HNZ17}, formulas (3.4)-(3.6)):
\begin{equation}\label{eq: compensation of Stratonovich to divergence}
\begin{aligned}
    \Delta(\mu_k) &= \mathbb{E} \left[\int_{0}^{T} x_k(t) \circ \d B_k^H(t)\right] \\
    &= \frac{1}{2} T^{2H} \big(1 -  \gamma(1,\mu_k T)\big) -  \mu_k^{-2H} \gamma(2H+1, \mu_k T) \Big(H-\frac{1}{2}\Big) \\&\qquad+ T H \mu_k^{1-2H} \gamma(2H, \mu_k T),
\end{aligned}
\end{equation}
where
\begin{equation}
\gamma(h, T) = \int_{0}^T \e^{-x} x^{h-1} \d x.\nonumber
\end{equation}
Moreover, since Stratonovich integral satisfies the rules of the first order calculus, we can easily evaluate (without need for rough path lift):
\begin{equation}
\int_{0}^{T} x_k(t) \circ \d x_k(t) = \frac{x^2_k(T) - x^2_k(0)}{2}.\nonumber
\end{equation}
This provides a simple pathwise formula for evaluation of the Skorokhod-type integral
\begin{equation}
    \int_{0}^{T} x_k(t) \d x_k(t) = \frac{x^2_k(T) - x^2_k(0)}{2} - \Delta(\mu_k),\nonumber
\end{equation}
and the corresponding formula for LSE
\begin{equation}
\hat{\lambda}_N = - \frac{\sum_{k=1}^{N}\alpha_k \left( \frac{x^2_k(T) - x^2_k(0)}{2} -  \Delta(\mu_k) \right) + \sum_{k=1}^{N}\alpha_k \beta_k \int_0^T x^2_k(t) \d t}{\sum_{k=1}^{N}\alpha_k^2 \int_0^T x^2_k(t) \d t}.\nonumber
\end{equation}

Unfortunately, $\mu_k = \alpha_k \lambda + \beta_k$ in this evaluation formula depend on the unknown value of parameter $\lambda$. A natural workaround here is to define new estimator (the pathwise LSE) as a solution to the following equation (in unknown $\Lambda$):

\begin{equation}\label{equation modified LSE}
    \Lambda = - \frac{\sum_{k=1}^{N}\alpha_k \left( \frac{x^2_k(T) - x^2_k(0)}{2} -  \Delta(\alpha_k \Lambda + \beta_k) \right) + \sum_{k=1}^{N}\alpha_k \beta_k \int_0^T x^2_k(t) \d t}{\sum_{k=1}^{N}\alpha_k^2 \int_0^T x^2_k(t) \d t}.
\end{equation}

\begin{Remark}
Although the theoretical LSE formula \eqref{LSE_theor1} does not explicitly contain the Hurst parameters $H$, its value is needed to evaluate the stochastic Skorokhod-type integral. Hence, it is present in defining equation for pathwise (computable) version of LSE. However, since we consider continuous-time setting, value of $H$ can be determined exactly from continuous observation of a single trajectory $\big\{x_k(t), t \in [0,T]\big\}$, as discussed above.
\end{Remark}

\begin{Remark}
If $H=1/2$ the pathwise LSE coincides with the theoretical LSE. This follows from the fact that the right-hand side of \eqref{equation modified LSE} is constant (in $\Lambda$) in this case, which is shown below.
\end{Remark}

\section{Existence and uniqueness of pathwise LSE}\label{sect: Existence and uniqueness of pathwise LSE}

We can not expect the existence and uniqueness of the positive solution to \eqref{equation modified LSE} in general. This can be seen on the simplest example of LSE for one-dimensional Ornstein-Uhlenbeck process (driven by a Wiener process)
\[
\hat{\lambda} = - \frac{\frac{x^2(T) - x^2(0)}{2} -  \frac{T}{2}}{\int_0^T x^2(t) \d t},
\]
which can provide negative estimates. We address this problem below by taking positive part of the solution and specifying, which one to choose if there are multiple solutions. To derive the conditions for the existence and uniqueness of the positive solution to equation \eqref{equation modified LSE}, we need to understand the properties of the function on its right-hand side. Set
\begin{equation}\label{eq: RHS function}
    R_N(\Lambda) = - \frac{\sum_{k=1}^{N}\alpha_k \left( \frac{x^2_k(T) - x^2_k(0)}{2} -  \Delta(\alpha_k \Lambda + \beta_k) \right) + \sum_{k=1}^{N}\alpha_k \beta_k \int_0^T x^2_k(t)\d t}{\sum_{k=1}^{N}\alpha_k^2 \int_0^T x^2_k(t)\d t},
\end{equation}
with function $\Delta$ defined in \eqref{eq: compensation of Stratonovich to divergence}. Direct calculations lead to the following formula for its first derivative (after cancelling some terms) 
\begin{equation}\label{eq: Delta 1. derivative}
\begin{aligned}
  \frac{\d \Delta}{\d \mu} (\mu) &=  H (1-2H) \big(T \mu^{-2H} \gamma(2H, \mu T) - \mu^{-2H-1} \gamma(2H+1, \mu T)  \big) \\
  &=  H (1-2H) \mu^{-2H-1} \int_0^{\mu T}\e^{-s}s^{2H-1}(\mu T - s) \d s\\
  &\begin{cases}
     > 0 \,\dots\, H < 1/2 \\
     = 0 \,\dots\, H = 1/2 \\
     < 0 \,\dots\, H > 1/2 \\
   \end{cases}, \quad \forall \mu > 0,
\end{aligned}
\end{equation}
and the second derivative
\begin{equation}\label{eq: Delta 2. derivative}
\begin{aligned}
  \frac{\d^2 \Delta}{\d \mu^2} (\mu) &=  H (2H-1) \mu^{-2H-2} \int_0^{\mu T}\e^{-s}s^{2H-1}\big(2 H \mu T - (2H+1)s\big) \d s\\
  &\begin{cases}
     < 0 \,\dots\, H < 1/2 \\
     = 0 \,\dots\, H = 1/2 \\
     > 0 \,\dots\, H > 1/2 \\
   \end{cases}, \quad \forall \mu > 0,
\end{aligned}
\end{equation}
in view of
\begin{equation*}
\begin{aligned}
   & \int_0^{\mu T}\e^{-s}s^{2H-1}\big(2 H \mu T - (2H+1)s\big) \d s  \\
    & \quad > \int_0^{\mu T 2H / (2H+1)}\e^{-\mu T 2H / (2H+1)}s^{2H-1}\big(2 H \mu T - (2H+1)s\big) \d s \\
    & \quad\quad +  \int_{\mu T 2H / (2H+1)}^{\mu T} \e^{-\mu T 2H / (2H+1)}s^{2H-1}\big(2 H \mu T - (2H+1)s\big) \d s = 0.
\end{aligned}
\end{equation*}
Derivatives of $R$ are immediate consequence of \eqref{eq: Delta 1. derivative} and \eqref{eq: Delta 2. derivative}
\begin{equation}\label{eq: R 1. derivative}
\begin{aligned}
  \frac{\d R_N}{\d \Lambda} (\Lambda) &=  \frac{\sum_{k=1}^{N}  \alpha^2_k \frac{\d \Delta}{\d \mu}(\alpha_k \Lambda + \beta_k)}{\sum_{k=1}^{N}\alpha_k^2 \int_0^T x^2_k(t)\d t}
  \begin{cases}
     > 0 \,\dots\, H < 1/2 \\
     = 0 \,\dots\, H = 1/2 \\
     < 0 \,\dots\, H > 1/2 \\
   \end{cases}, \quad \forall \Lambda > 0,
\end{aligned}
\end{equation}
and
\begin{equation}
\begin{aligned}\nonumber
  \frac{\d^2 R_N}{\d \Lambda^2} (\Lambda) &=  \frac{\sum_{k=1}^{N}  \alpha^3_k \frac{\d^2 \Delta}{\d \mu^2}(\alpha_k \Lambda + \beta_k)}{\sum_{k=1}^{N}\alpha_k^2 \int_0^T x^2_k(t)\d t}
  \begin{cases}
     < 0 \,\dots\, H < 1/2 \\
     = 0 \,\dots\, H = 1/2 \\
     > 0 \,\dots\, H > 1/2 \\
   \end{cases}, \quad \forall \Lambda > 0.
\end{aligned}
\end{equation}
Moreover, since
\[
\lim_{\mu \to 0_+}\Delta(\mu) = \frac{1}{2}T^{2H},
\]
we obtain
\begin{equation}\label{R lim to 0}
    \lim_{\Lambda \to 0_+} R_N(\Lambda) =: R_N(0) \in \mathbb{R}.
\end{equation}
Direct calculations yield
\begin{equation}
    \lim_{\Lambda \to \infty} R_N(\Lambda) =
       \begin{cases}
        \infty \,\dots\, H < 1/2 \\
        R_N^\infty \in \mathbb{R} \,\dots\, H \geq 1/2
   \end{cases}\nonumber
\end{equation}
and
\begin{equation}
    \lim_{\Lambda \to \infty} R_N(\Lambda) - \Lambda = -\infty, \quad \forall H \in (0,1).\nonumber
\end{equation}
The above properties of $R_N$ lead us to the following conclusion:

\begin{Theorem}\label{thm: existence and uniqueness of modified lambda}
Fix $N$ and consider the function $R_N$ defined in \eqref{eq: RHS function} with $R_N(0)$ defined in \eqref{R lim to 0} and let $H \in (0,1)$. Sufficient condition for the existence and uniqueness of the positive solution to equation \eqref{equation modified LSE}, which can be rephrased as $R_N(\Lambda) = \Lambda$, is
\begin{equation}
    R_N(0)>0.\nonumber
\end{equation}
If $H\geq 1/2$, this condition is also necessary.
\end{Theorem}
Note that the condition $R_N(0) > 0$ can be easily verified from the data by direct evaluation of $R_N$ at $\Lambda = 0$ (if $\beta_k$ are positive) or by calculating the limit for $\Lambda \to 0_{+}$ (if $\beta_k$ are zero). We are now in a position to define the pathwise LSE correctly:
\begin{equation}\label{definition of modified LSE}
\tilde{\lambda}_N :=
  \begin{cases}
     \text{positive solution to \eqref{equation modified LSE}, if it exists and is unique};\\
      0, \text{if there is no positive solution to \eqref{equation modified LSE}}; \\
     \text{the greater solution to \eqref{equation modified LSE}, if there are two positive solutions}. \\
   \end{cases}
\end{equation}

\begin{Remark}
The choice of the greater solution in the third case ensures that $\frac{\d R_N}{\d \Lambda}(\tilde{\lambda}_N) < 1$, which corresponds to the limiting behavior of $\frac{\d R_N}{\d \Lambda}$ in the neighborhood of the true value $\lambda$, detailed in Lemma \ref{lem: R_N derivative uniform conv}
\end{Remark}

\begin{Remark}
Monotonicity and convexity (if $H>1/2$) or concavity (if $H<1/2$), respectively, ensure that the Newton-Raphson numerical method is well suited for numerical approximation of $\tilde{\lambda}_N$.
\end{Remark}


\section{Strong consistency} \label{sect: Strong consistency}
\subsection{Strong consistency of theoretical LSE}\label{subsect: Strong consistency of theoretical LSE}
We start with strong consistency of the theoretical LSE $\hat{\lambda}_N$ defined in \eqref{LSE_theor1}, because it will be helpful to show strong consistency of the pathwise LSE in the sequel. For simplicity, we assume throughout this section that
\begin{equation}
	x_k(0) = 0, \quad k=1,2,\ldots\nonumber
\end{equation}

The self-similarity of a fBm enables us to study the effect of different drifts $\mu_k$ on the distributions of the processes $x_k$. Its combination with the (long-span) asymptotic behavior of a real-valued fractional Ornstein-Uhlenbeck process, studied in \cite{HNZ17},  forms the strategy of the proof of the strong consistency (in space) of $\hat{\lambda}_N$.

Using integration-by-parts in the formula~\eqref{OU} it is possible to show that the solutions $\{x_k(t),t\geq 0\}$ can be expressed as
\begin{equation}\label{OUD}
x_k(t)=\left(B^H_k(t)-\mu_k\e^{-\mu_k t}\int_0^t \e^{\mu_k s}B^H_k(s)\d s\right),\ t\geq 0,
\end{equation}
for any $k\in\N$. This can also be verified by direct substitution of~\eqref{OUD} into the integral form of equation~\eqref{SDE1}.

Let $\{\tilde x_k(t),t\geq 0\}$ be the solution to the equation with unit drift
\begin{equation}\label{rceK}
\begin{array}{rcl}
{\rm d} \tilde x_k(t)&=& -\tilde x_k(t)\d t+ \d B^H_k(t),\quad t>0,\ k\in\N,\\
\tilde x_k(0)&=&0,
\end{array}
\end{equation}
given by the formula
\begin{equation}
\tilde x_k(t)=\int_0^t \e^{-(t-s)} \d B^H_k(t),\quad t\geq 0,\nonumber
\end{equation}
or
\begin{equation}\label{OUDK}
\tilde x_k(t)= \left(B^H_k(t)-\int_0^t \e^{-(t-s)}B^H_k(s)\d s\right),\quad t\geq 0,
\end{equation}
for any $k\in\N$.

Using the self-similarity of a fractional Brownian motion in the expression~\eqref{OUDK} we get the equality of distributions
\begin{equation}\label{selfS}
\{x_k(t),t\geq 0\}\stackrel{d}{=} \left\{\frac{1}{\mu_k^H}\tilde x_k(\mu_k\,t),t\geq 0\right\}
\end{equation}
for any $k\in\N$.

Moreover, let $\{y(t),t\geq 0\}$ be the stationary and ergodic solution to~\eqref{rceK} (for $k=1$) given by a formula
\[y(t)=\e^{-t}y_0+\int_0^t \e^{-(t-s)} \d B^H_1(t),\ t\geq 0.\]
Then 
\[\E y(t) = 0,\ {\rm var}\, \big(y(t)\big)= H\Gamma(2H)\]
for any $t\geq 0$ and
\begin{equation}\label{erg}
\frac{1}{T}\int_0^T y^2(t)\d t \xrightarrow[T\rightarrow\infty]{} H\Gamma(2H) \quad \hbox{a.s. and in }L^1.
\end{equation}

\begin{Lemma}\label{Ekonv}
The following convergence
\[\E\left[\frac{1}{T}\int_0^T \big(\tilde x^2_1(t)-y^2(t)\big)\d t\right] \xrightarrow[T\rightarrow\infty]{} 0\]
holds.
\end{Lemma}

\begin{proof}
We have
\begin{align*}
&\E\left[\frac{1}{T}\int_0^T \big(\tilde x^2_1(t)-y^2(t)\big)\d t\right] \\&\qquad=\E\left[\frac{1}{T}\Big( \int_0^T 2y(t)\big(\tilde x_1(t)-y(t)\big) \d t + \int_0^T \big(\tilde x_1(t)-y(t)\big)^2 \d t\Big )\right] \\&\qquad= \frac{1}{T}\left (\int_0^T -2\e^{-t} \E\big[y_0 \,y(t)\big]\d t + \int_0^T \e^{-2t}\E[y_0^2] \d t\right ) \\&\qquad = \frac{1}{T}\int_0^{1} -2\e^{-t} \E\big[y_0 \,y(t)\big]\d t + \frac{1}{T}\int_{1}^T -2\e^{-t} \E\big[y_0 \,y(t)\big]\d t\\&\qquad\qquad+ H\Gamma(2H)\frac{1}{T}\int_0^T \e^{-2t}\d t.
\end{align*}
Obviously, the first and third term tend to zero as $T$ goes to infinity. Also the second term tends to zero due to~\cite{KM19}, Lemma~3.1, and the limit properties of 
\[\E \big[y_0 \,y(t)\big]=\frac{1}{2} 2H(2H-1)t^{2H-2}+O(t^{2H-4}),\ t\rightarrow \infty,\]
for $H\neq 1/2$ (see \cite{ChKM03}, Theorem 2.3) and
\[\E \big[y_0 \,y(t)\big]=\frac{1}{2} \e^{-t},\ t\geq 0,\] for the Wiener case $H=1/2$.
\end{proof}

Define the weak asymptotic equivalence $\asymp$ for sequences $\{a_N\}_{N\in\N}, \{b_N\}_{N\in\N}$ by the relation
\begin{equation}\label{def: asymp}
\biggl(a_N\asymp b_N\biggr) \quad \equiv \quad \biggl(\exists\, 0<l<L<\infty \text{ such that } l\, b_N\leq a_N\leq L\, b_N, \forall N\in\N\biggr).
\end{equation}
Note that this relation is weaker than the equivalence relation $\sim$ defined in \eqref{def: sim}.

Recall the formula \eqref{LSE_theor2} (suitable for probabilistic analysis) and denote 
\[C_N= \sum_{k=1}^N \alpha_k\int_0^T x_k(t)\d B^H_k(t),\]
and
\[D_N= \sum_{k=1}^N \alpha_k^2\int_0^T x_k^2(t)\d t,\]
so that 
\[\hat \lambda_N = \lambda -\frac{C_N}{D_N}.\]

For the needs of Lemma~\ref{Lem: a.s. convergence} we explore the limit behavior of the variance of $\frac{C_N}{\E D_N}$ and $\frac{D_N}{\E D_N}$. Denote for a fixed $T$ 
\[T_k:= \mu_k T,\ k\in\N.\]

\begin{Proposition}\label{str_h}
The expectations of $C_N$ and $D_N$ satisfy
\[ \E C_N=0\quad\hbox{and}\quad \E D_N\asymp \sum_{k=1}^N \frac{\alpha_k^2}{\mu_k^{2H}}.\]
\end{Proposition}

\begin{proof}
The first part is obvious. By Lemma~\ref{Ekonv} 
\[\E\left[\frac{1}{T}\int_0^{T} \tilde x^2_1(t)\d t\right] \xrightarrow[T\rightarrow\infty]{} H\Gamma(2H), \]
thus the sequence $\left\{1/T_k \int_0^{T_k} \E \tilde x^2_1(t)\d t\right\}_{k \in \N}$ is positive and bounded. Using~\eqref{selfS} and substitution we get
\begin{align*}
\E D_N &= \sum_{k=1}^N \alpha_k^2\int_0^T \E x_k^2(t)\d t= \sum_{k=1}^N \frac{\alpha_k^2}{\mu_k^{2H}}\int_0^T \E\tilde x_k^2(\mu_k t)\d t \\&=  \sum_{k=1}^N \frac{\alpha_k^2}{\mu_k^{2H}}\int_0^T  \E\tilde x_1^2(\mu_k t)\d t = \sum_{k=1}^N \frac{\alpha_k^2}{\mu_k^{2H}}T \frac{1}{T_k}\int_0^{T_k}  \E\tilde x_1^2(t)\d t \asymp \sum_{k=1}^N \frac{\alpha_k^2}{\mu_k^{2H}}.
\end{align*}
\end{proof}

To study the time asymptotics of $\var \Big(\int_0^T\tilde{x}_1^2(t)\d t\Big)$, we utilize the results from \cite{HNZ17}. Denote
\[F_T:=\int_0^T \tilde{x}_1(t) \d B^H_1(t) = \int_0^T\Big(\int_0^t \e^{-(t-s)} \d B^H_1(s)\Big)\d B^H_1(t).\]
By (3.9) in~\cite{HNZ17} this equals

\begin{equation}\label{HNZrovn}
F_T=\left[\Big(\frac{\tilde{x}^2_1(T)}{2}-\E \frac{\tilde{x}^2_1(T)}{2}\Big) + \int_0^T \big(\tilde{x}^2_1(t)- \E \tilde{x}^2_1(t)\big)\d t\right].
\end{equation}

Further define
\renewcommand{\arraystretch}{1.3}
$$Q_H(T):=\left\{
\begin{array}{ccl}
\frac{1}{T} & , & H<\frac{3}{4},\\
\frac{1}{T\log T} & , & H=\frac{3}{4},\\
\frac{1}{T^{4H-2}} & , & H>\frac{3}{4}.
\end{array}\right.$$
\renewcommand{\arraystretch}{1}
Then by \cite{HNZ17}, Lemma 17, there exists a constant $0<C_H<\infty$ such that
\begin{equation}\label{HNZkonv}
\E\big[Q_H(T)F^2_T\big]\xrightarrow[T\rightarrow\infty]{} C_H.
\end{equation}
As a consequence, we obtain:

\begin{Lemma}\label{konv_det}
$Q_H(T)\,\var \Big(\int_0^T\tilde{x}_1^2(t)\d t\Big) \xrightarrow[T\rightarrow\infty]{} C_H$.
\end{Lemma}

\begin{proof} In the first step we show that $\var \big(\tilde{x}_1^2(T)\big)\xrightarrow[T\rightarrow\infty]{} \var y_0^2$.

\noindent
Since $\tilde{x}_1(T)= y(T) - \e^{-T} y_0$, using the triangle inequality we obtain
\begin{align*}
\left(\sqrt{\E y^2(T)} - \e^{-T}\sqrt{\E y_0^2}\right)^2 \leq &\E \tilde{x}_1^2(T)\leq \left(\sqrt{\E y^2(T)} + \e^{-T}\sqrt{\E y_0^2}\right)^2\\
H\Gamma(2H) (1 - \e^{-T})^2 \leq &\E \tilde{x}_1^2(T)\leq H\Gamma(2H)(1 + \e^{-T})^2.
\end{align*}
Similarly,
\[\E y_0^4 \,(1 - \e^{-T})^4 \leq \E \tilde{x}_1^4(T)\leq \E y_0^4 \,(1 + \e^{-T})^4.\]
Therefore
\begin{align*}
&\overbrace{\E y_0^4 \,(1 - \e^{\!-T})^4 - \big(H\Gamma(2H) (1 + \e^{\!-T})^2\big)^2}^{\xrightarrow[T\rightarrow\infty]{}\,\var y_0^2} \leq \var \tilde{x}_1^2(T)\\&\hspace{5cm}\leq \underbrace{\E y_0^4 \,(1 + \e^{\!-T})^4 - \big(H\Gamma(2H) (1 - \e^{-T})^2\big)^2}_{\xrightarrow[T\rightarrow\infty]{}\,\var y_0^2}.
\end{align*}

\noindent
In the second step we apply \eqref{HNZkonv} and the first step to equality \eqref{HNZrovn} and yield the statement.
\end{proof}

\renewcommand{\arraystretch}{2}
\begin{Proposition}\label{rozptyl}
The variances of $C_N$ and $D_N$ satisfy

\[
\left\{
\begin{array}{lllll}
H<\frac{3}{4} & \Rightarrow & \var C_N\asymp \sum_{k=1}^N \dfrac{\alpha_k^2}{\mu_k^{4H-1}} & , & \var D_N\asymp \sum_{k=1}^N \dfrac{\alpha_k^4}{\mu_k^{4H+1}},\\
H=\frac{3}{4} & \Rightarrow & \var C_N\asymp \sum_{k=1}^N \dfrac{\alpha_k^2}{\mu_k^2}\log T_k & , & \var D_N\asymp \sum_{k=1}^N \dfrac{\alpha_k^4}{\mu_k^4}\log T_k,\\
H>\frac{3}{4} & \Rightarrow & \var C_N\asymp \sum_{k=1}^N \dfrac{\alpha_k^2}{\mu_k^2} & , & \var D_N\asymp \sum_{k=1}^N \dfrac{\alpha_k^4}{\mu_k^4}.
\end{array}
\right.
\]

\end{Proposition}
\renewcommand{\arraystretch}{1}

\begin{proof} Using again (3.9) in \cite{HNZ17}, \eqref{selfS} and the fact that $\{\tilde x_k(t),t\geq 0\}$ have the same law for all $k\in \N$, we get equality of laws:
\[\int_0^T x_k(t) \d B^H_k(t) \stackrel{d}{=} \dfrac{1}{\mu_k^{2H}}\int_0^{T_k} \tilde x_1(t) \d B^H_1(t)\]
for any $k\in\N$. This implies that
\begin{align*}
\var C_N&=\var\Big(\sum_{k=1}^N \alpha_k\int_0^T x_k(t) \d B^H_k(t)\Big) = \sum_{k=1}^N \alpha_k^2\,\var \Big(\int_0^T x_k(t)\d B^H_k(t) \Big) \\&= \sum_{k=1}^N \frac{\alpha_k^2}{\mu_k^{4H}}\,\var \Big(\int_0^{T_k} \tilde x_1(t) \d B^H_1(t)\Big) = \sum_{k=1}^N \frac{\alpha_k^2}{\mu_k^{4H}}\frac{1}{Q_H(T_k)}\E\big[ Q_H(T_k)F^2_{T_k}\big],
\end{align*}
where the last equality follows from the definition of $F_T$. By~\eqref{HNZkonv} the sequence $\left\{\E\big[ Q_H(T_k)F^2_{T_k}\big]\right\}_{k \in \N}$ is bounded and together with the definition of $Q_H$ the first part of the statements follows.

For the variance of $D_N$, apply \eqref{selfS} to get
\begin{align*}
\var D_N &= \var\Big(\sum_{k=1}^N \alpha_k^2\int_0^T x_k^2(t)\d t\Big) = \sum_{k=1}^N \alpha_k^4 \,\var \Big( \int_0^T x_k^2(t)\d t \Big) \\&= \sum_{k=1}^N \frac{\alpha_k^4}{\mu_k^{4H}} \,\var \Big( \int_0^T \tilde x_k^2(\mu_k t)\d t \Big) \\&= \sum_{k=1}^N \frac{\alpha_k^4}{\mu_k^{4H+2}} \frac{1}{Q_H(T_k)} Q_H(T_k) \,\var \Big( \int_0^{T_k} \tilde x_1^2(t)\d t\Big),
\end{align*}
and Lemma~\ref{konv_det} to obtain that the sequence $\left\{Q_H(T_k) \,\var \Big(\int_0^{T_k} \tilde x_1^2(t)\d t\Big)\right\}_{k \in \N}$is bounded. The second parts of the statements are again concluded by the definition of $Q_H$.
\end{proof}

Combining Proposition~\ref{str_h} and Proposition~\ref{rozptyl} we obtain the asymptotic formulas for variances of $\frac{C_N}{\E D_N}$ and $\frac{D_N}{\E D_N}$.
\begin{equation}\label{rozptyl_podilu}
\left\{
\begin{array}{llll}
H<\frac{3}{4}: & \var \dfrac{C_N}{\E D_N}\asymp \dfrac{\sum_{k=1}^N \dfrac{\alpha_k^2}{\mu_k^{4H-1}}}{\left(\sum_{k=1}^N \frac{\alpha_k^2}{\mu_k^{2H}}\right)^{\! 2}} & , & \var \dfrac{D_N}{\E D_N}\asymp \dfrac{\sum_{k=1}^N \dfrac{\alpha_k^4}{\mu_k^{4H+1}}}{\left(\sum_{k=1}^N \frac{\alpha_k^2}{\mu_k^{2H}}\right)^{\! 2}},\\[7mm]
H=\frac{3}{4}: & \var \dfrac{C_N}{\E D_N}\asymp \dfrac{\sum_{k=1}^N \dfrac{\alpha_k^2}{\mu_k^2}\log T_k}{\left(\sum_{k=1}^N \frac{\alpha_k^2}{\mu_k^{3/2}}\right)^{\! 2}} & , & \var \dfrac{D_N}{\E D_N}\asymp \dfrac{\sum_{k=1}^N \dfrac{\alpha_k^4}{\mu_k^4}\log T_k}{\left(\sum_{k=1}^N \frac{\alpha_k^2}{\mu_k^{3/2}}\right)^{\! 2}},\\[7mm]
H>\frac{3}{4}: & \var \dfrac{C_N}{\E D_N}\asymp \dfrac{\sum_{k=1}^N \dfrac{\alpha_k^2}{\mu_k^2}}{\left(\sum_{k=1}^N \frac{\alpha_k^2}{\mu_k^{2H}}\right)^{\! 2}} & , & \var \dfrac{D_N}{\E D_N}\asymp \dfrac{\sum_{k=1}^N \dfrac{\alpha_k^4}{\mu_k^4}}{\left(\sum_{k=1}^N \frac{\alpha_k^2}{\mu_k^{2H}}\right)^{\! 2}}.
\end{array}
\right.
\end{equation}

Note that if for some $c,\delta>0$ and for all $N\in\N$
\begin{equation}\label{odhad_var}
\var \frac{C_N}{\E D_N}\leq c \,N^{-\delta}\ \hbox{ and }\ \var \frac{D_N}{\E D_N}\leq c \,N^{-\delta}
\end{equation}
holds true then application of Lemma~\ref{Lem: a.s. convergence} \big($C_N/\E D_N$ and $(D_N-\E D_N)/ \E D_N$ are in the second Wiener chaos, see \eqref{eq: C_N D_N in CH_2} in Appendix\big) provides us with
$$\frac{C_N}{\E D_N}\xrightarrow[N\rightarrow\infty]{} 0\quad \hbox{ and }\quad \frac{D_N}{\E D_N}\xrightarrow[N\rightarrow\infty]{} 1\ \hbox{ a.s.}$$
so that the strong consistency of the estimate
$$\hat{\lambda}_N = \lambda -\dfrac{\frac{C_N}{\E D_N}}{\frac{D_N}{\E D_N}}\xrightarrow[N\rightarrow\infty]{} \lambda\ \hbox{ a.s.}$$
follows.

It is difficult to verify the sufficient conditions \eqref{odhad_var} in general and it should be done on case-by-case basis. However, if we adopt assumption \eqref{SPDE_ev}, often satisfied for parabolic SPDEs, the calculations of asymptotic variances simplify substantially.  

\begin{Theorem}\label{thm:cons_theor_1}
Let $B\equiv 0$, i.e. $\beta_k=0$ for all $k\in\N$, and the power growth \eqref{SPDE_ev} is satisfied for~$\alpha_k$. Then $\hat{\lambda}_N$ is strongly consistent for any $m_1,d\in\N$.
\end{Theorem}

\begin{proof}
Denote the growth exponent of $\alpha_k$ by $M_1:=2m_1/d$. Then using formula~\eqref{rozptyl_podilu} we obtain

\begin{align}
\var\frac{C_N}{\E D_N}
&\asymp\var \frac{D_N}{\E D_N}\nonumber\\\label{eq:var speed B=0} 
&\asymp\left\{
\begin{array}{lll}
\dfrac{\sum_{k=1}^N k^{M_1(3-4H)}}{\left(\sum_{k=1}^N k^{2M_1(1-H)}\right)^{\! 2}}\asymp N^{-(M_1+1)} & , & \!H<\frac{3}{4},\\[4mm]
\dfrac{\sum_{k=1}^N \log (k^{M_1}T)}{\left(\sum_{k=1}^N k^{1/2 M_1}\right)^{\! 2}}\asymp N^{-(M_1+1)}\log N & , & \!H=\frac{3}{4},\\
\dfrac{N}{\left(\sum_{k=1}^N k^{2M_1(1-H)}\right)^{\! 2}}\asymp N^{-\big(4M_1(1-H)+1\big)} & , & \!H>\frac{3}{4}.
\end{array}\right.
\end{align}

Thus~\eqref{odhad_var} is satisfied for any $m_1,d$, so the conclusion follows from Lemma~\ref{Lem: a.s. convergence}.
\end{proof}

\begin{Theorem}\label{thm:cons_theor_2}
Assume the power growth of eigenvalues $\alpha_k$ and $\beta_k$ given in~\eqref{SPDE_ev}: 
\[
\alpha_k \sim k^{2m_1/d} \hbox{  and  } \beta_k \sim k^{2m_2/d}.
\]
If 
\begin{itemize}
\item[(i)] $m_1>m_2$ then $\hat{\lambda}_N$ is strongly consistent for any $m_1,m_2,d\in\N$,
\item[(ii)] $m_1<m_2$ and 
\[m_1>\left\{\begin{array}{lll}
-\frac{d}{4}+\frac{1}{2}m_2 & , & H<\frac{3}{4},\\[2mm]
-\frac{d}{4}+m_2(2H-1) & , & H\geq \frac{3}{4},
\end{array}\right.\]
then $\hat{\lambda}_N$ is strongly consistent.
\end{itemize}
\end{Theorem}

\begin{proof}
\begin{itemize}
\item[(i)] Applying again formula~\eqref{rozptyl_podilu} we obtain the same estimates as in the proof of Theorem~\ref{thm:cons_theor_1} thus the conclusion follows.
\item[(ii)] For conciseness, we consider the canonical case (all the involved series have power growth) only. The special cases, where some series in the expectation of $D_N$ or the variances of $C_N, D_N$ have logarithmic growth, have to be analyzed separately. Nevertheless, they are covered by~(ii) as well.

\noindent
For $\var \frac{D_N}{\E D_N}$ we get the same estimates as in the proof of Theorem~\ref{thm:cons_theor_1} but with $M_1$ substituted by $M_2:=2m_2/d$. For $\var \frac{C_N}{\E D_N}$ formula~\eqref{rozptyl_podilu} yields

\begin{align}
&\var \dfrac{C_N}{\E D_N}\nonumber\\[4mm]
&\quad \asymp\left\{
\begin{array}{l@{\!}ll}
\dfrac{\sum_{k=1}^N k^{2M_1-M_2(4H-1)}}{\left(\sum_{k=1}^N k^{2(M_1-HM_2)}\right)^{\! 2}}\asymp N^{-2M_1+M_2-1} & , & H<\frac{3}{4},\\[4mm]
\dfrac{\sum_{k=1}^N k^{2M_1+2M_2}\log(k^{M_2}T)}{\left(\sum_{k=1}^N k^{2M_1-3/2 M_2}\right)^{\! 2}}\asymp N^{-2M_1+M_2-1}\log N & , & H=\frac{3}{4},\\[4mm]
\dfrac{\sum_{k=1}^N k^{2(M_1-M_2)}}{\left(\sum_{k=1}^N k^{2(M_1-H M_2)}\right)^{\! 2}}\asymp N^{-2\big(M_1+M_2(1-2H)\big)-1} & , & H>\frac{3}{4}.
\end{array}\right. \label{eq:var speed B nonzero}
\end{align}

These estimates give us the required relationship between $m_1,m_2$ and $d$ to meet~\eqref{odhad_var}. The rest follows from Lemma~\ref{Lem: a.s. convergence}. 
\end{itemize}
\end{proof}

\begin{Remark}
The case $m_1=m_2$ is reduced to the case of Theorem~\ref{thm:cons_theor_1}.
\end{Remark}

\begin{Remark}
The expressions for $\var \frac{C_N}{\E D_N}$ given in formulas \eqref{eq:var speed B=0} and \eqref{eq:var speed B nonzero} show the speed of the convergence of $\hat{\lambda}_N$ to $\lambda$.
\end{Remark}

\subsection{Strong consistency of pathwise LSE}\label{subsect: Strong consistency of pathwise LSE}

In this subsection, strong consistency (as $N \to \infty$) of $\tilde{\lambda}_N$ is demonstrated. We continue to denote the true value of the unknown parameter by $\lambda$. Recall that the theoretical LSE satisfies $\hat{\lambda}_N = R_N(\lambda)$. Hence, if $\hat{\lambda}_N$ is strongly consistent (e.g. if the conditions of Theorem \ref{thm:cons_theor_1} or Theorem \ref{thm:cons_theor_2} are satisfied), we have:
\begin{equation}
    \lim_{N \to \infty} R_N(\lambda) = \lambda.\nonumber
\end{equation}
To show that the pathwise LSE $\tilde{\lambda}_N = R_N(\tilde{\lambda}_N)$ is arbitrarily close to $\lambda$ for sufficiently large $N$, the limiting behavior of derivatives of $R_N$ is studied. The following elementary observation will be useful in the sequel:

\begin{Lemma}\label{lem: a_k b_k}
Let $a_k, b_k \geq 0$ for $k = 1,2,\dots$, $ \lim_{k \to \infty} b_k = b \in (0,\infty)$, and $\sum_{k=1}^{\infty} a_k = \infty$. Then
\begin{equation}
 \lim_{N \to \infty} \frac{\sum_{k=1}^{N} a_k b_k}{\sum_{k=1}^{N} a_k b} = 1.\nonumber
\end{equation}
\end{Lemma}

\begin{proof}
It is a simple consequence of the Stolz--Cesaro theorem.
\end{proof}

Next lemma bounds the derivatives of $R_N$ on a compact neighborhood of $\lambda$ from above by a constant smaller than one (for sufficiently large N):

\begin{Lemma}\label{lem: R_N derivative uniform conv}
Let the power growth condition~\eqref{SPDE_ev} be satisfied and $H\neq 1/2$. Further assume:
\begin{equation}\label{cond1}   
    \lim_{k \to \infty} \frac{\beta_k}{\alpha_k} = c \in [0,\infty],
\end{equation}
\begin{equation}\label{cond2}
    \sum_{k=1}^{\infty}\frac{\alpha_k^2}{(\alpha_k+\beta_k)^{2H}} = \infty.
\end{equation}
Then, almost surely, if we take arbitrary  $0< \Lambda_L < \Lambda_U < \infty$ we have
\begin{equation}\label{eq:  R_N derivative uniform conv}
\begin{aligned}
   & \lim_{N \to \infty} \sup_{\Lambda \in [\Lambda_L, \Lambda_U]} \left[\frac{1}{1-2H}\frac{\d R_N}{\d \Lambda} (\Lambda) \right] =\frac{(\lambda + c)^{2H}}{(\Lambda_L + c)^{2H}},\\
   & \lim_{N \to \infty} \inf_{\Lambda \in [\Lambda_L, \Lambda_U]} \left[\frac{1}{1-2H}\frac{\d R_N}{\d \Lambda} (\Lambda) \right] =\frac{(\lambda + c)^{2H}}{(\Lambda_U + c)^{2H}},
\end{aligned}
\end{equation}
where the right-hand sides equal one if $c = \infty$.
\end{Lemma}

\begin{proof}
Recall the formula for the derivative of $R_N$ in \eqref{eq: R 1. derivative} and the notation from previous subsection $D_N = \sum_{k=1}^{N}\alpha_k^2 \int_0^T x^2_k(t)\d t$. It follows from the proof of Proposition \ref{str_h} that
\[
    \mathbb{E}D_N = \sum_{k=1}^N \frac{\alpha_k^2}{\mu_k^{2H}} T \varphi(\mu_k T),
\]
where
\[
    \varphi(\tau) = \frac{1}{\tau}\int_{0}^{\tau}\mathbb{E} \tilde{x}^2_1(t)\d t \xrightarrow[\tau \to \infty]{} H \Gamma(2H).
\]
Use \eqref{eq: R 1. derivative} to write (for any $\Lambda > 0$):
\begin{equation}
\begin{aligned}
&\frac{\d R_N}{\d \Lambda} (\Lambda)  =  \frac{\sum_{k=1}^{N} \alpha^2_k \frac{\d \Delta}{\d \mu}(\alpha_k \Lambda + \beta_k)}{D_N} = \frac{\mathbb{E}D_N}{D_N}  \frac{\sum_{k=1}^{N}  \alpha^2_k \frac{\d \Delta}{\d \mu}(\alpha_k \Lambda + \beta_k)}{\sum_{k=1}^N \frac{\alpha_k^2}{\mu_k^{2H}} T \varphi(\mu_k T)}.
\end{aligned}
\end{equation}
Formula \eqref{eq: Delta 1. derivative} together with $\lim_{T \to \infty} \gamma(h, T) = \Gamma(h)$ yield
\begin{equation}\label{eq: lim a^2H d Delta }
 \lim_{\mu \to \infty} \mu^{2H} \frac{\d \Delta}{\d \mu} (\mu) =  T H (1-2H) \Gamma(2H).
\end{equation}
Consequently, we can factorize
\begin{equation}
\begin{aligned}
&\frac{\d R_N}{\d \Lambda} (\Lambda) =  \frac{\mathbb{E}D_N}{D_N} U_N  V_N(\Lambda)  W_N(\Lambda),\nonumber
\end{aligned}
\end{equation}
with

\begin{align*}
   U_N &=  \frac{\sum_{k=1}^N \frac{\alpha_k^2}{\mu_k^{2H}} T H \Gamma(2H)}{\sum_{k=1}^N \frac{\alpha_k^2}{\mu_k^{2H}} T \varphi(\mu_k T)}, \\
   V_N(\Lambda) &= \frac{\sum_{k=1}^{N}  \alpha^2_k \frac{\d \Delta}{\d \mu}(\alpha_k \Lambda + \beta_k)}{\sum_{k=1}^{N}  \frac{\alpha^2_k}{(\alpha_k \Lambda + \beta_k)^{2H}}   T H (1-2H) \Gamma(2H)},\\
   W_N(\Lambda) &= \frac{\sum_{k=1}^{N}  \frac{\alpha^2_k}{(\alpha_k \Lambda + \beta_k)^{2H}} T H (1-2H) \Gamma(2H)}{\sum_{k=1}^N \frac{\alpha_k^2}{\mu_k^{2H}} T H \Gamma(2H)}.
\end{align*}

Recall that condition~\eqref{SPDE_ev} implies (cf. proofs of Theorem~\ref{thm:cons_theor_1} and Theorem~\ref{thm:cons_theor_2})
\[
    \lim_{N \to \infty} \frac{D_N}{\mathbb{E}D_N} = 1 \quad \text{almost surely}.
\]
Assumptions of Lemma \ref{lem: a_k b_k} with
\[
a_k = \frac{\alpha_k^2}{\mu_k^{2H}}, \quad b_k=T \varphi(\mu_k T), \quad \text{and } b = T H \Gamma(2H),
\]
are guaranteed by conditions \eqref{cond2} and \eqref{lim mu_k}. In result,
\[
\lim_{N \to \infty} U_N = 1.	
\]
Denote (motivated by Lemma \ref{lem: a_k b_k})
\[
a_k(\Lambda) = \frac{\alpha^2_k}{(\alpha_k \Lambda + \beta_k)^{2H}}, \quad b_k(\Lambda)= (\alpha_k \Lambda + \beta_k)^{2H}  \frac{\d \Delta}{\d \mu}(\alpha_k \Lambda + \beta_k)
\] 
and 
\[ b = T H (1-2H)  \Gamma(2H).
\]
Fix arbitrary  $0<\Lambda_L< \Lambda_U<\infty$.  Obviously, due to \eqref{lim mu_k} and \eqref{eq: lim a^2H d Delta },
\[
\sup_{\Lambda \in [\Lambda_L,\Lambda_U]}\big|b_k(\Lambda) - b\big| \xrightarrow[k \to \infty]{} 0,
\]
and, due to \eqref{cond2}
\[
\inf_{\Lambda \in [\Lambda_L,\Lambda_U]}\sum_{k=1}^{N}a_k(\Lambda) = \sum_{k=1}^{N}a_k(\Lambda_U) \xrightarrow[N \to \infty]{} \infty.
\]
Similarly to proof of Lemma \ref{lem: a_k b_k} we can show
\[
\lim_{N \to \infty} \sup_{\Lambda \in [\Lambda_L,\Lambda_U]} \big|V_N(\Lambda)-1\big| = \lim_{N \to \infty} \sup_{\Lambda \in [\Lambda_L,\Lambda_U]} \left|\frac{\sum_{k=1}^{N} a_k(\Lambda) b_k(\Lambda)}{\sum_{k=1}^{N} a_k(\Lambda) b} -1\right|  =  0.	
\]

For the last term we can use the upper bound:
\[
  \sup_{\Lambda \in [\Lambda_L,\Lambda_U]} \left[\frac{W_N(\Lambda)}{1-2H}\right] = \frac{\sum_{k=1}^{N}  \frac{\alpha^2_k}{(\alpha_k \Lambda_L + \beta_k)^{2H}} }{\sum_{k=1}^N \frac{\alpha_k^2}{(\alpha_k \lambda + \beta_k)^{2H}}}.
\]
Assume first $\lim_{k \to \infty} \frac{\beta_k}{\alpha_k} = c \in [0,\infty)$. Then
\[
    \frac{\sum_{k=1}^{N}  \frac{\alpha^2_k}{(\alpha_k \Lambda_L + \beta_k)^{2H}} }{\sum_{k=1}^N \frac{\alpha_k^2}{(\alpha_k \lambda + \beta_k)^{2H}}} = \frac{\sum_{k=1}^{N}  \alpha^{2-2H}_k\big(\Lambda_L + \frac{\beta_k}{\alpha_k}\big)^{-2H}} {\sum_{k=1}^{N}  \alpha^{2-2H}_k\big(\lambda + \frac{\beta_k}{\alpha_k}\big)^{-2H}}.
\]
Application of Lemma \ref{lem: a_k b_k} with
\[
a_k = \alpha^{2-2H}_k, \quad b_k = \Big(\Lambda + \frac{\beta_k}{\alpha_k}\Big)^{\!-2H}, \quad b = (\Lambda + c)^{-2H},
\]
and $\Lambda = \Lambda_L$ or $\Lambda = \lambda$, respectively, and where
\[
\sum_{k=1}^{\infty} \alpha^{2-2H}_k \geq \sum_{k=1}^{\infty}\frac{\alpha_k^2}{(\alpha_k+\beta_k)^{2H}} = \infty,
\]
results in
\begin{equation}
    \lim_{N \to \infty} \frac{\sum_{k=1}^{N}  \frac{\alpha^2_k}{(\alpha_k \Lambda_L + \beta_k)^{2H}} }{\sum_{k=1}^N \frac{\alpha_k^2}{(\alpha_k \lambda + \beta_k)^{2H}}} = \lim_{N \to \infty} \frac{\sum_{k=1}^{N}  \alpha^{2-2H}_k(\Lambda_L + c)^{-2H}} {\sum_{k=1}^{N}  \alpha^{2-2H}_k(\lambda + c)^{-2H}} = \frac{(\lambda + c)^{2H}}{(\Lambda_L + c)^{2H}}.\nonumber
\end{equation}
Next, assume instead $c = \infty$, i.e. $\lim_{k \to \infty}\frac{\alpha_k}{\beta_k}= 0$. We can proceed similarly to previous case
\begin{equation}
\lim_{N \to \infty} \frac{\sum_{k=1}^{N}  \frac{\alpha^2_k}{(\alpha_k \Lambda_L + \beta_k)^{2H}} }{\sum_{k=1}^N \frac{\alpha_k^2}{(\alpha_k \lambda + \beta_k)^{2H}}} = \lim_{N \to \infty} \frac{\sum_{k=1}^{N}  \frac{\alpha^2_k}{\beta_k^{2H}}\big(\frac{\alpha_k}{\beta_k} \Lambda_L + 1\big)^{-2H}} {\sum_{k=1}^{N}  \frac{\alpha^2_k}{\beta_k^{2H}}\big(\frac{\alpha_k}{\beta_k} \lambda + 1\big)^{-2H}} = 1.\nonumber
\end{equation}
To summarize, we have
\[
\lim_{N \to \infty} \sup_{\Lambda \in [\Lambda_L,\Lambda_U]} \left[\frac{W_N(\Lambda)}{1-2H}\right] =
\begin{cases}
\frac{(\lambda + c)^{2H}}{(\Lambda_L + c)^{2H}} \quad &\text{if } c \in [0,\infty),\\
1 \quad &\text{if }c= \infty.\\
\end{cases}
\]
The limit for infimum can be obtained accordingly, since
\[
  \inf_{\Lambda \in [\Lambda_L,\Lambda_U]} \left[\frac{W_N(\Lambda)}{1-2H}\right] = \frac{\sum_{k=1}^{N}  \frac{\alpha^2_k}{(\alpha_k \Lambda_U + \beta_k)^{2H}} }{\sum_{k=1}^N \frac{\alpha_k^2}{(\alpha_k \lambda + \beta_k)^{2H}}}.
\]
\end{proof}

\begin{Theorem}\label{thm: strong consistency modifie LSE}
Assume that the eigenvalues $\alpha_k$ and $\beta_k$ satisfy the power growth conditions~\eqref{SPDE_ev} and the asymptotic conditions~\eqref{cond1}, \eqref{cond2} and \eqref{lim mu_k}. Further assume that the theoretical LSE $\hat{\lambda}_N$ is strongly consistent (see e.g. Theorem~\ref{thm:cons_theor_1} and Theorem~\ref{thm:cons_theor_2}). Then the pathwise LSE is strongly consistent as well:
\begin{equation}\label{eq: strong consistency modifie LSE}
 \lim_{N \to \infty} \tilde{\lambda}_N = \lambda \quad \text{almost surely}.
\end{equation}
Moreover,  $\tilde{\lambda}_N$ and $\hat{\lambda}_N$ have the same speed of convergence (up to a constant):
\begin{equation}\label{eq: speed modified LSE and theoretical LSE}
    \lim_{N \to \infty} \frac{|\lambda - \hat{\lambda}_N|}{|\lambda - \tilde{\lambda}_N|} =  2H \quad \text{almost surely}.
\end{equation}
\end{Theorem}

\begin{proof}
Recall that if $H=1/2$, the theorem holds true trivially, because $\tilde{\lambda}_N = \hat{\lambda}_N$  in this case. In the rest of the proof, consider $H \neq 1/2$.
First, take $\Lambda_L < \lambda$ so that
\[
(1-2H) \frac{(\lambda + c)^{2H}}{(\Lambda_L + c)^{2H}} < 1.
\]
and arbitrary $\Lambda_U > \lambda$. Lemma \ref{lem: R_N derivative uniform conv} guarantees that, for almost all $\omega$ and some $\Psi >0$,
\[
    \inf_{\Lambda \in [\Lambda_L, \Lambda_U]} \left[1- \frac{\d R_N}{\d \Lambda} (\Lambda)\right] >  \Psi,
\]
for sufficiently large $N$. Fix $\epsilon > 0$ such that $\Lambda_L < \lambda - \epsilon <  \lambda + \epsilon < \Lambda_U$ and denote $\delta = \epsilon  \Psi$. The almost sure convergence $\lim_{N \to \infty} R_N(\lambda) = \lambda$, which is guaranteed by strong consistency of the theoretical LSE, enables us to take $N$ large enough in order to
\[
  \big|\lambda - R_N(\lambda)\big| < \delta.
\]
Combination of the lower bound on derivative of $\Lambda - R_N(\Lambda)$ and its closeness to zero at $\lambda$ implies existence of its root inside $(\lambda-\epsilon,  \lambda+\epsilon)$, which is unique therein and it equals $\tilde{\lambda}_N$. This proves \eqref{eq: strong consistency modifie LSE}.

\bigskip
To demonstrate \eqref{eq: speed modified LSE and theoretical LSE}, note that for sufficiently large $N$
\[
\frac{|\lambda - \hat{\lambda}_N|}{|\lambda - \tilde{\lambda}_N|} = \left| 1-  \frac{\d R_N}{\d \Lambda} (\Lambda^0_N) \right|,
\]
where $\Lambda^0_N$ is between $\tilde{\lambda}_N$ and $\lambda$ and therefore converges to $\lambda$ almost surely. Next, fix arbitrary $\epsilon > 0$ and take $\Lambda_L < \lambda <\Lambda_U$ so that the right-hand sides of \eqref{eq:  R_N derivative uniform conv} are close enough to one and
\[
\sup_{\Lambda \in [\Lambda_L, \Lambda_U]} \left| \frac{\d R_N}{\d \Lambda} (\Lambda) - (1-2H) \right| < \epsilon,
\]
for sufficiently large $N$. In particular, for  $N$ large enough we have almost surely
\[
\left| \frac{\d R_N}{\d \Lambda} (\Lambda^0_N) - (1-2H) \right| < \epsilon.
\]
This proves the desired almost sure convergence
\[
 \lim_{N \to \infty} \frac{|\lambda - \hat{\lambda}_N|}{|\lambda - \tilde{\lambda}_N|} =\lim_{N \to \infty} \left| 1-  \frac{\d R_N}{\d \Lambda} (\Lambda^0_N) \right| = \big|1-(1-2H)\big| = 2H.
\]
\end{proof}

\begin{Remark}
Immediate, and possibly surprising, consequence of \eqref{eq: speed modified LSE and theoretical LSE} is that the pathwise LSE performs asymptotically better than the theoretical LSE whenever $H>1/2$.
\end{Remark}

\subsection{Example}\label{subsect: Example}
As an example of the use of Theorem \ref{thm: strong consistency modifie LSE}, consider a stochastic heat equation with additional mean reversion, with a distributed noise autocorrelated in time and with the Dirichlet boundary condition. Formally, we can write this equation as

\begin{equation}\label{eq: HE_example}
\begin{aligned}
\frac{\partial u}{\partial t}(t,\xi) &= \lambda_1 \; \Delta_\xi u(t,\xi) - \lambda_2 \; u(t,\xi) + \eta^H(t,\xi), \quad \text{ for } (t,\xi) \in \mathbb{R}_+ \times \mathcal{O},   \\
u(t,\xi) &= 0, \quad \text{ for } (t,\xi)  \in \mathbb{R}_+ \times \partial \mathcal{O},   \\
u(0,\xi)&=0, \quad \text{ for } \xi \in \mathcal{O}  ,
\end{aligned}
\end{equation}
where $\Delta_\xi$ is Laplace operator in variable $\xi$,  $\mathcal{O} \subset \mathbb{R}^d$ is a $d$-dimensional bounded domain with smooth boundary $\partial \mathcal{O}$, $\lambda_1 \geq 0$ is the diffusivity parameter determining the speed of diffusion of a modeled substance along the domain $\mathcal{O}$, $\lambda_2 \geq 0$ is the rate of mean reversion related to the speed of reversion of the amount  of substance to the zero mean. The process $\{\eta^H(t,\xi);\, t\geq 0, \xi \in \mathcal{O}\}$ is a Gaussian noise, which is fractional (autocorrelated) in time with Hurst parameter $H\in(0,1)$ and white in space. Formally it can be modelled as the differential of a cylindrical fractional Brownian motion $\d B^H(t)$ (see Appendix \ref{Appendix:cfBm} for details). 
\bigskip
 
This formal equation can be rewritten rigorously as a stochastic evolution equation~\eqref{sEE} in $V = L^2(\mathcal{O})$, with $A = \Delta|_{Dom(A)}$, where $Dom(A) = H^2(\mathcal{O}) \cap H^1_0(\mathcal{O})$, being  Dirichlet Laplace operator defined on a standard Sobolev space (cf. \cite{Shubin2001}) and with $B$ being the identity operator on $L^2(\mathcal{O})$, should $\lambda_1$ be the estimated parameter (or vice-versa, should we estimate $\lambda_2$). Consequently, this equation is diagonalizable with eigenvalues satisfying the power growth condition \eqref{SPDE_ev}. 
\bigskip

Firstly, assume we want to estimate the diffusivity parameter $\lambda_1$. If $\lambda_2 = 0$, we have a stochastic heat equation with $\alpha_k \sim k^{2/d}$, and $\beta_k = 0$. Theorem~\ref{thm:cons_theor_1} guarantees strong consistency of the theoretical LSE $\hat{\lambda}_N$. If  $\lambda_2 > 0$, there is an additional mean reversion in the process and we have $\alpha_k \sim k^{2/d}$, and $\beta_k \sim 1$. We can apply Theorem~\ref{thm:cons_theor_2} (i) to get strong consistency of the theoretical LSE $\hat{\lambda}_N$. In both cases, Theorem~\ref{thm: strong consistency modifie LSE} implies strong consistency of the pathwise LSE $\tilde{\lambda}_N$ with the speed of convergence (in terms of RMSE) given by 
\begin{equation}\label{eq:speed HE lamba_1}
\sqrt{\E(\tilde{\lambda}_N - \lambda)^2} \asymp	\left\{
\begin{array}{lll}
 N^{-(1/d+1/2)} & , & H<\frac{3}{4},\\[1mm]
 N^{-(1/d+1/2)}\sqrt{\log N} & , & H=\frac{3}{4},\\
 N^{-\big(4(1-H)/d+1/2\big)} & , & H>\frac{3}{4}.
\end{array}\right. 
\end{equation}	 

Secondly, let $\lambda_2$ be the parameter of interest. If $\lambda_1 = 0$ we have $\alpha_k \sim 1$ and $\beta_k = 0$, and the existence condition \eqref{eq: existence cond} is violated. Hence, we consider the case $\lambda_1 > 0$, when the process combines diffusion and mean reversion. For the eigenvalues, we have $\alpha_k \sim 1$, and $\beta_k \sim k^{2/d}$. Theorem~\ref{thm:cons_theor_2} (ii) provides us with strong consistency of $\hat{\lambda}_N$ if $d>2$ for $H<3/4$ and $d>4(2H-1)$ otherwise. To show strong consistency of the pathwise LSE by Theorem~\ref{thm: strong consistency modifie LSE}, we need in addition that $d \geq 4H$ so that~\eqref{cond2} is satisfied. If the conditions for consistency are satisfied, the speed of convergence is
\begin{equation}\label{eq:speed HE lamba_2}
\sqrt{\E(\tilde{\lambda}_N - \lambda)^2} \asymp	\left\{
\begin{array}{lll}
 N^{1/d-1/2} & , & H<\frac{3}{4},\\[1mm]
 N^{1/d-1/2}\sqrt{\log N} & , & H=\frac{3}{4},\\[1mm]
 N^{(4H-2)/d-1/2} & , & H>\frac{3}{4}.
\end{array}\right. 
\end{equation}


\section{Simulation study}\label{sect: Simulation study}

To illustrate the actual performance of the studied pathwise LSE we perform a Monte Carlo analysis for two specific equations -- a 1D heat equation (the simplest setting) and a 2D heat equation with coupling with surroundings (a model popular e.g. in oceanography). All simulations were performed in statistical software R. 

\subsection{1D stochastic heat equation}
This equation has very simple structure and serves as an easy-to-undersand toy-model suitable for illustration of the theory above. We simulate a stochastic heat equation with zero boundary condition from subsection~\ref{subsect: Example} on the line segment (0,1), with true diffusivity $\lambda_1 = 1$ being a parameter to be estimated ($\lambda = \lambda_1$) and without additional mean reversion ($\lambda_2 = 0$). In particular, we set 
\begin{itemize}
	\item $d = 1$, $\mathcal{O} = (0,1)$ and time horizon $T = 1$,
	\item $\bigl(A f\bigr)(\xi) = \frac{\partial^2 f}{\partial \xi^2}, \ \xi \in (0,1), \quad B \equiv 0$, 
	\item $e_k(\xi) = \sin(k \pi \xi), \quad \alpha_k = (k \pi)^2, \quad \beta_k = 0 \quad \forall k \in \mathbb{N}$,
	\item two scenarios for Hurst index: $H \in \{0.2, 0.8\}$, 
	\item two scenarios for initial condition: (i) $x_k(0) = 0$ for all $k \geq 1$ (zero IC), and (ii) $x_1(0) = 10, \ x_2(0) = 5, \ x_3(0) = 2, \ x_k(0) = 0, \ k \geq 4$ (non-zero IC).
\end{itemize}
The observed Fourier modes take the following form 
\[x_k(t) = \int_0^1 u(t,\xi) \sin(k \pi \xi) \d \xi.\]
For simulations we used the Spectral Galerkin Method (see \cite{Lord_etAl_2014}, chapter 10.7.), which is very efficient in our setting (space discretization provides system of independent linear stochastic ODEs, which, in addition, can be solved explicitly). Heat maps of four simulated sample solutions for various scenarios are shown in Figure \ref{fig:sample_solution1D}.

\begin{figure}
  \includegraphics[scale=0.5, keepaspectratio]{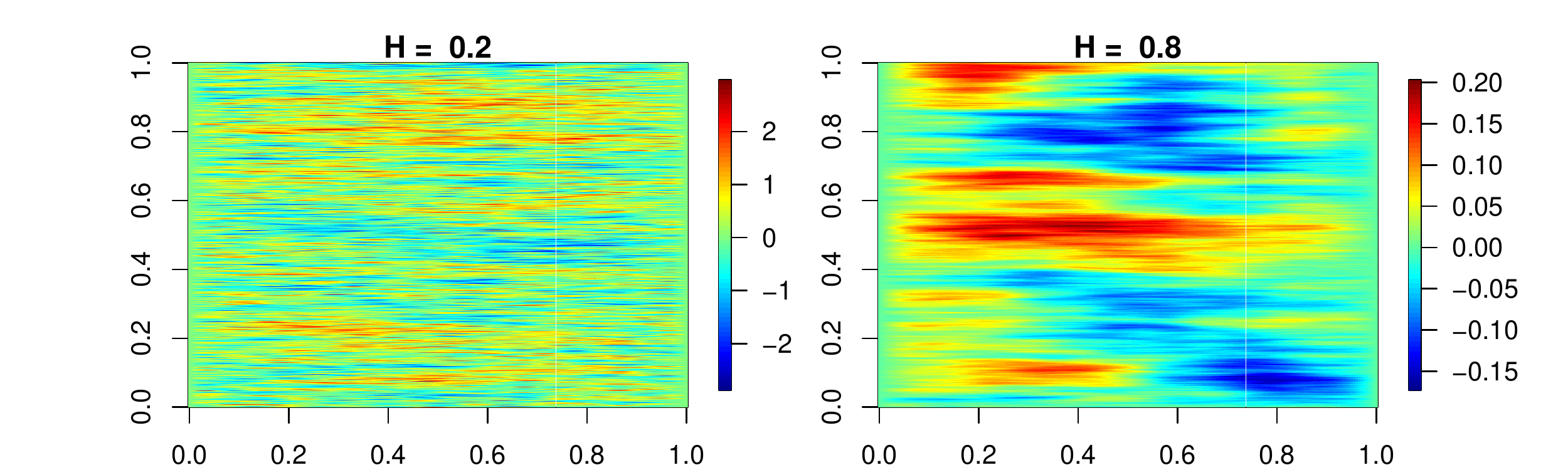} 
	\includegraphics[scale=0.5, keepaspectratio]{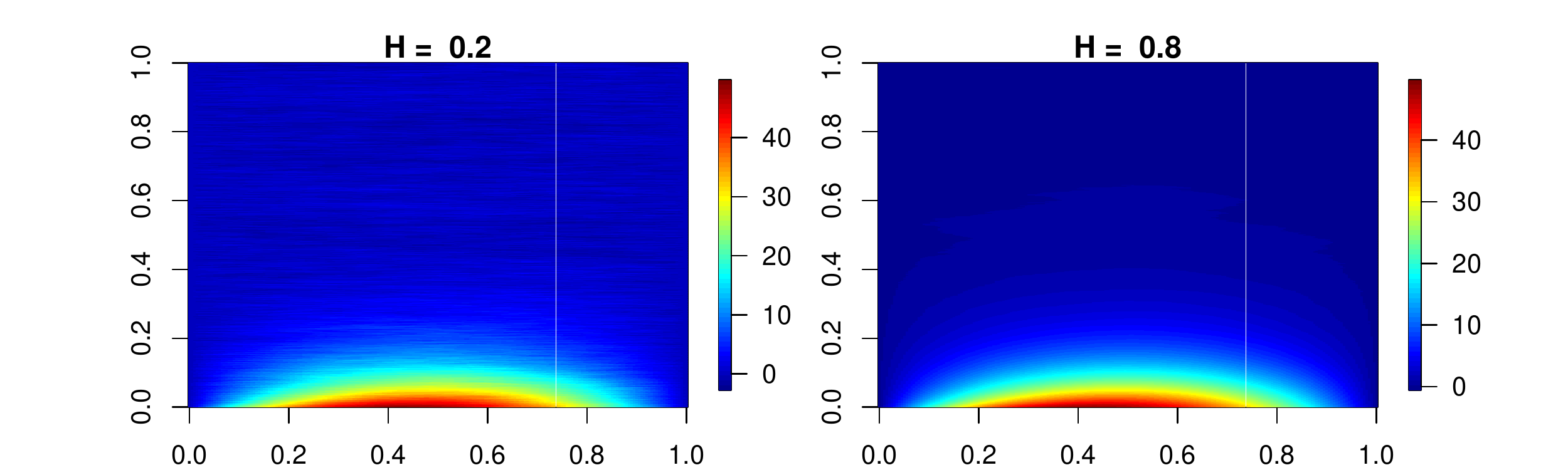}
  \caption{Heat maps of sample solutions of 1D stochastic heat equation with zero (the upper row) and non-zero (the lower row) initial conditions. Horizontal axis represents the physical space domain and vertical axis the time domain.}
  \label{fig:sample_solution1D}
\end{figure}

We test performance of the pathwise LSE, defined in \eqref{definition of modified LSE}, and compare it with the weighted MCE, an alternative estimator defined and studied in \cite{Kriz2020}, cf. formulas (3.21)\,--\,(3.23) and Theorem 3.2 therein. Recall that the weighted MCE is strongly consistent and asymptotically normal with increasing number of observed modes. 

Simulation results were obtained from 200 runs for each scenarios producing sample of 200 estimates for each setting. Figure~\ref{fig:convergence_quantiles1D} illustrates the convergence of the estimates to the true value of the parameter with increasing number of Fourier modes by showing selected sample quantiles. It confirms the convergence for all settings. In addition, non-zero initial condition significantly improves the performance of the pathwise LSE. Figure~\ref{fig:convergence_RMSE1D} shows convergence of the root mean square error (RMSE) for samples of the two types of estimators (pathwise LSE and weighted MCE) and its comparison with the theoretical speed of convergence for RMSE from \eqref{eq:speed HE lamba_1} (applicable only for solutions with zero initial condition). In case of zero initial condition, RMSE for the pathwise LSE and the weighted MCE practically coincide and their speed of convergence is consistent with the theoretical one. For solutions with far-from-zero initial conditions RMSE for the pathwise LSE is close to zero even for small number of Fourier modes, whereas the weighted MCE converges significantly slower. Figure \ref{fig:QQplot1D} shows comparison of sample quantiles of pathwise LSE (40 Fourier modes considered) with quantiles of normal distribution including 95\% confidence envelope. Although not proved theoretically, these Q-Q plots suggest asymptotic normality of the pathwise LSE for all settings considered.

\begin{figure}
  \includegraphics[scale=0.4,keepaspectratio]{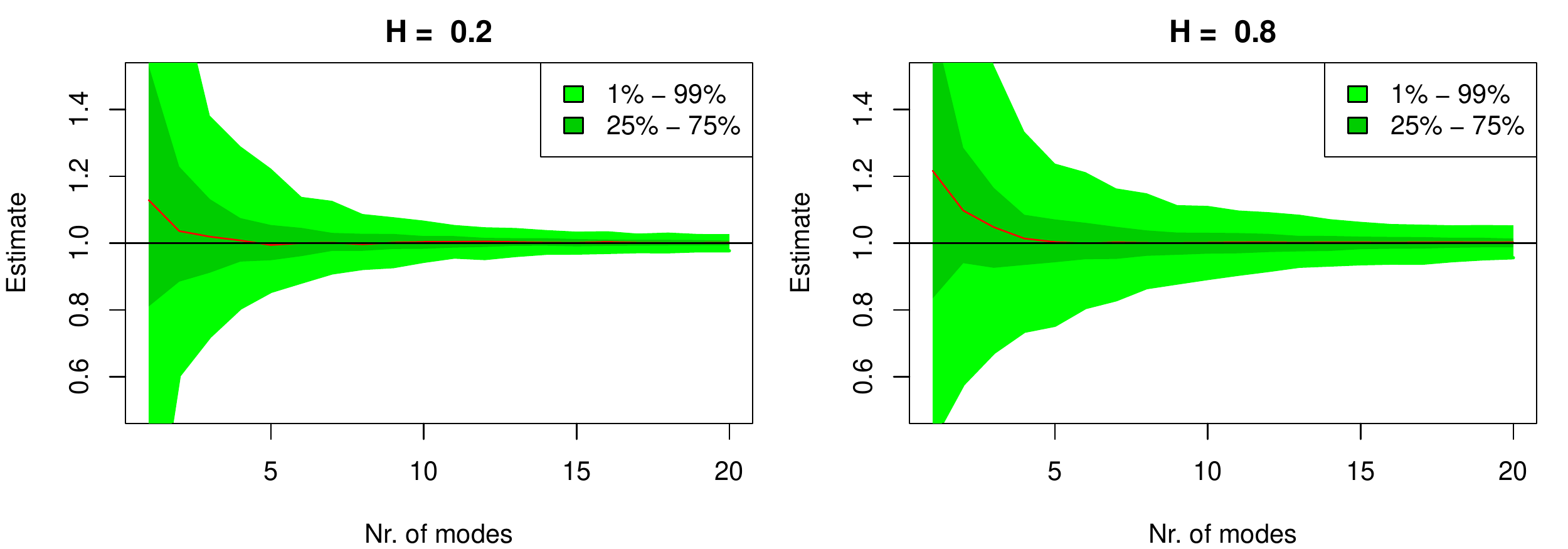}
  \includegraphics[scale=0.4,keepaspectratio]{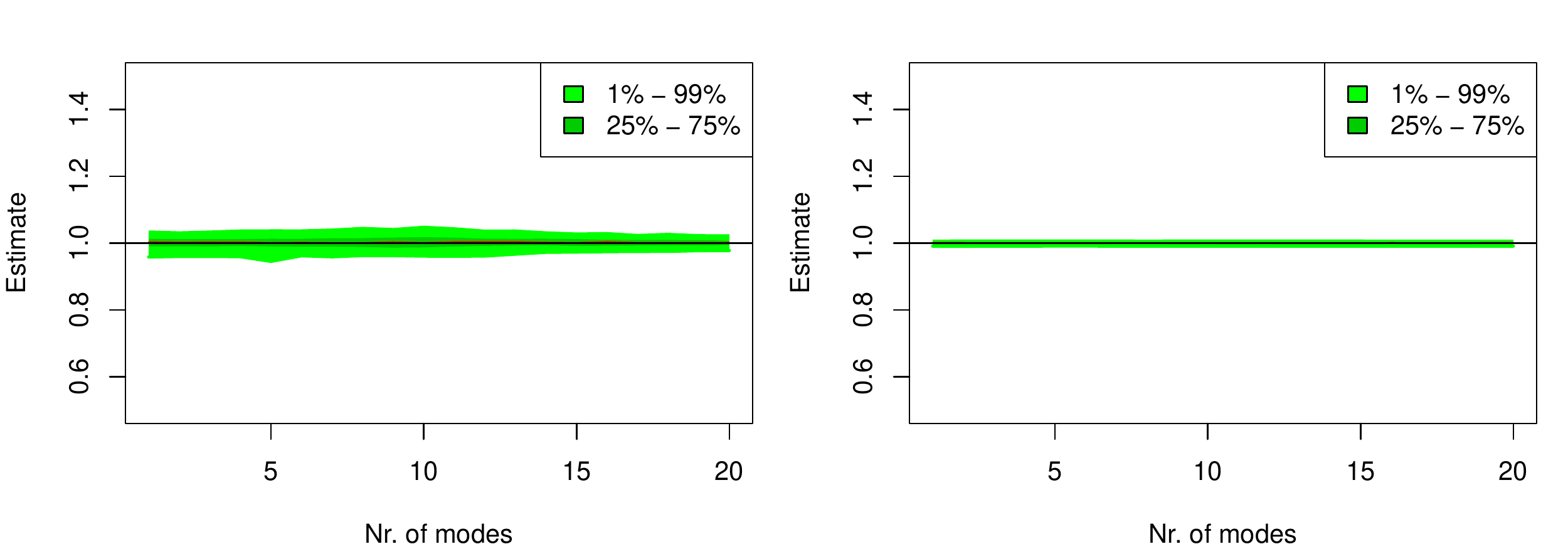}	
  \caption{Sample quantile ranges (green areas) and median (red line) of the simulated estimates (200 samples), true value is 1.\\
	Left column: Hurst index = 0.2 (rough case); right column: Hurst index = 0.8 (smooth case).\\
	First row: solutions with zero init. cond.; second row: solutions with non-zero init. cond.
	}
  \label{fig:convergence_quantiles1D}
\end{figure}

\begin{figure}
	\includegraphics[scale=0.4,keepaspectratio]{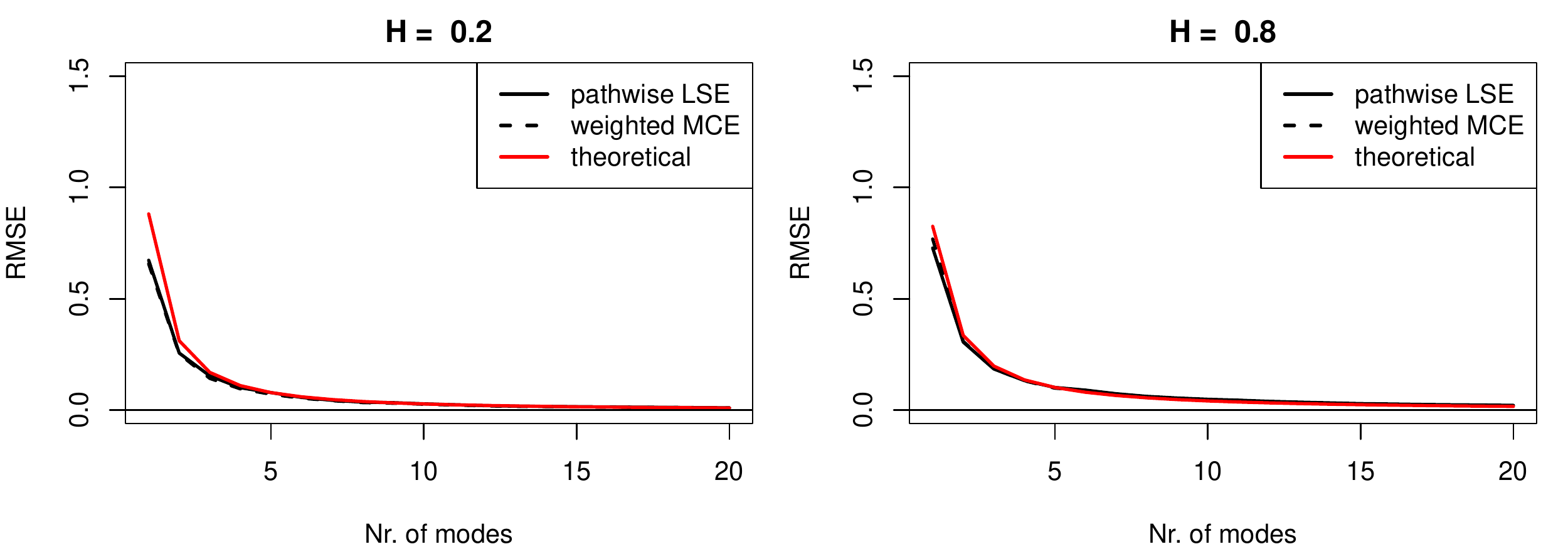}
  \includegraphics[scale=0.4,keepaspectratio]{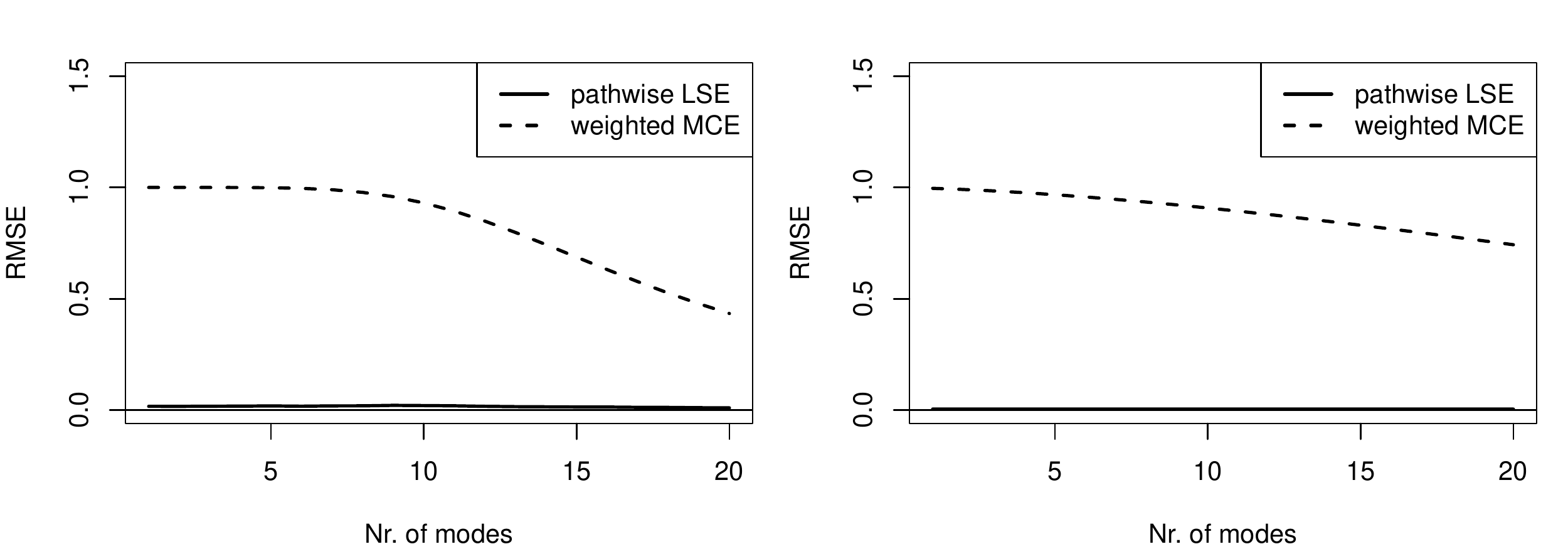}
  \caption{Root mean square error (RMSE) for pathwise LSE and weighted MCE and its theoretical speed of convergence from \eqref{eq:speed HE lamba_1} (applicable only for zero initial condition).\\
	Left column: Hurst index = 0.2 (rough case); right column: Hurst index = 0.8 (smooth case).\\
	First row: solutions with zero init. cond.; second row: solutions with non-zero init. cond.
	}
	\label{fig:convergence_RMSE1D}
\end{figure}

\begin{figure}
	\includegraphics[scale=0.4,keepaspectratio]{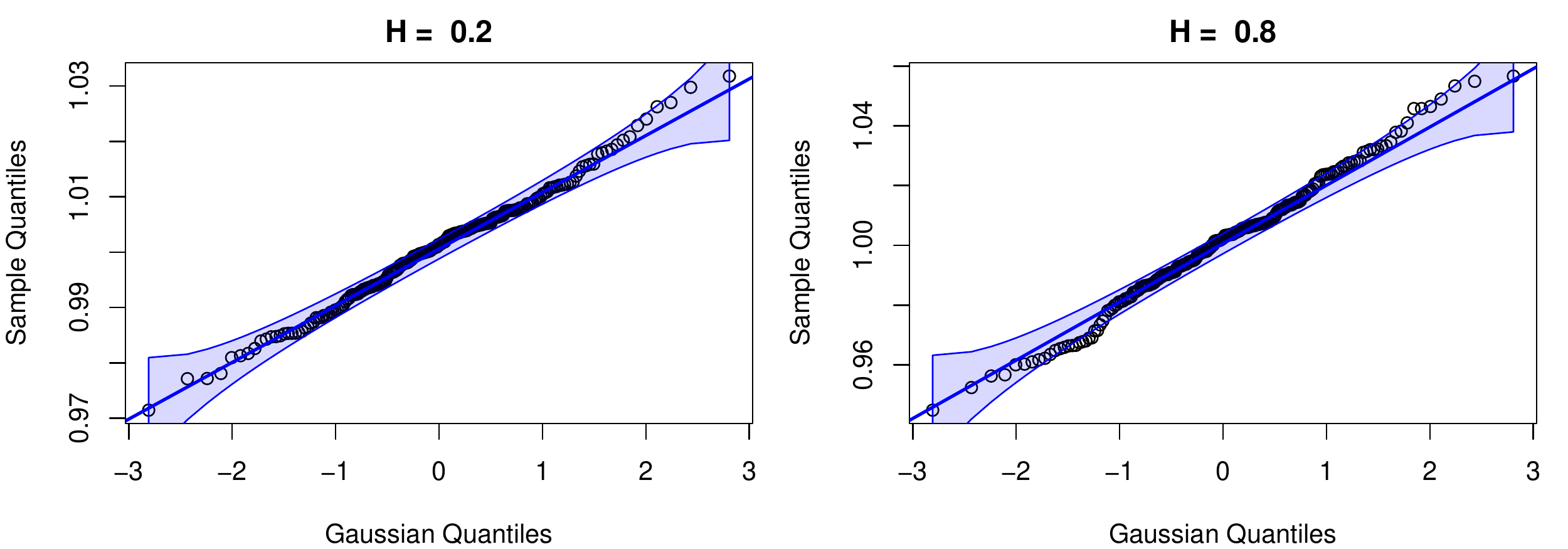}
	\includegraphics[scale=0.4,keepaspectratio]{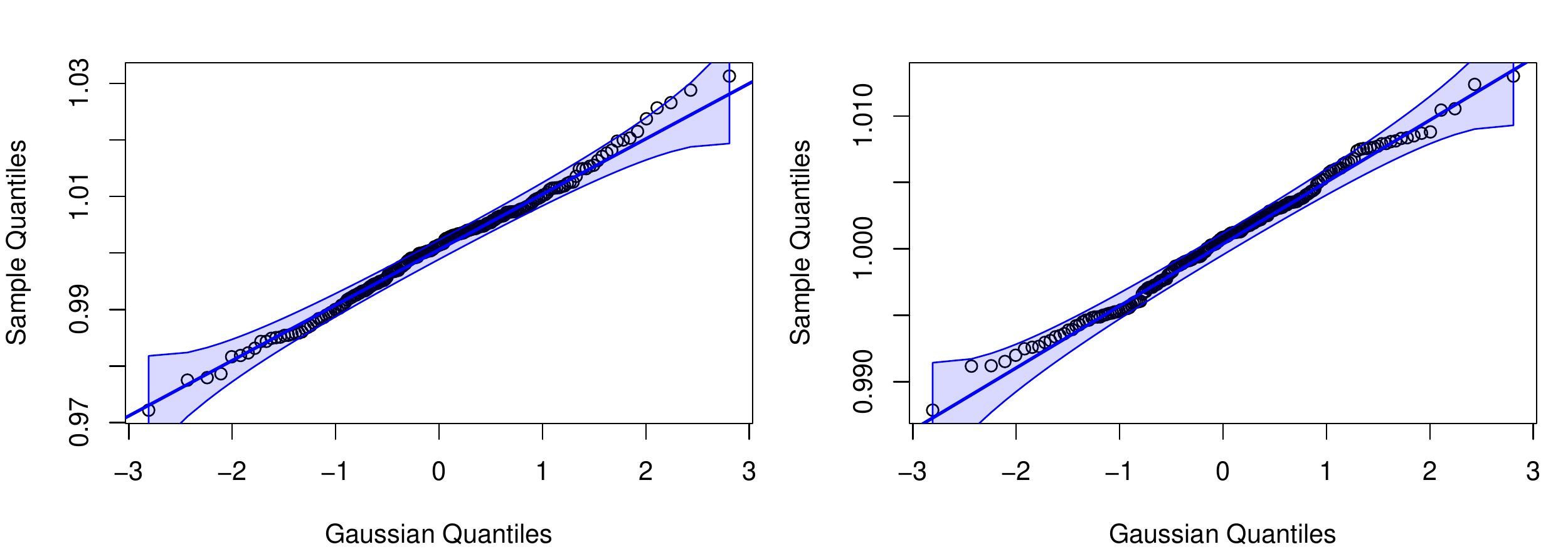}
  \caption{Q-Q plot for sample of 200 pathwise LSE estimates (calculated from 40 Fourier modes) including 95\% confidence envelope for normal distribution.\\
	Left column: Hurst index = 0.2 (rough case); right column: Hurst index = 0.8 (smooth case).\\
	First row: solutions with zero init. cond.; second row: solutions with non-zero init. cond.	
	}
  \label{fig:QQplot1D}
\end{figure}

\subsection{2D stochastic heat equation with coupling with surroundings}
Equations of this type are popular (not only) in physical oceanography, where they model fluctuations of the temperature of the top layer of the ocean around its long-time average, see \cite{PiterbargRozovsky1996} or \cite{LototskyRozovsky2017}. These equations fall into the setting of Example in~\ref{subsect: Example} with $d=2$ (a two-dimensional domain being the surface of the ocean), where the 2D Laplace operator models thermal conduction along the water surface, the mean-reversion term reflects atmosphere and ocean coupling and the noise term represents random forcing (e.g. vertical heat fluxes due to turbulent environment) -- for more details see \cite{PiterbargRozovsky1996}. The authors in \cite{PiterbargRozovsky1996} and \cite{LototskyRozovsky2017} (Section 6.1.1) discuss estimation of parameters in these equations based on observations in spectral domain under the assumption that the noise process is white in time and space (generated by a standard cylindrical Brownian motion). In contrast, we study parameter estimation assuming more general noise, which can be autocorrelated in time (generated by a cylindrical fractional Brownian motion).

In particular, we simulate equation \eqref{eq: HE_example} with the following setting:
\begin{itemize}
	\item $d = 2$, $\mathcal{O} = (0,1)^2$ and time horizon $T = 1$,
	\item $\lambda_1 = 1$ (the parameter to be estimated),  $\lambda_2 = 1$, 
	\item Hurst index $H = 0.3$, zero initial condition.
\end{itemize}
We used again the Spectral Galerkin Method with eigenfuctions $\sin(k_1 \pi \xi_1) \times \sin(k_2 \pi \xi_2)$  and corresponding eigenvalues $1 + (k_1^2 + k_2^2) \pi^2$, where $k_1, k_2 = 0,1,2,\dots$ and $(\xi_1,\xi_2) \in (0,1)^2$.  

The simulated distribution of the heat along the water surface at specific time instants is shown in Figure~\ref{fig:sample_solution2D}. Results of the Monte Carlo study of the behavior of the pathwise LSE obtained from 200 runs are depicted in Figure~\ref{fig:sim_results2D}. They confirm our theoretical findings on the almost-sure convergence of the pathwise LSE and its speed of convergence (in terms of RMSE). Moreover, Q-Q plot indicate its asymptotic normality, although we do not provide theoretical reasoning (for conciseness).

\begin{figure}
	\includegraphics[scale=0.5, keepaspectratio]{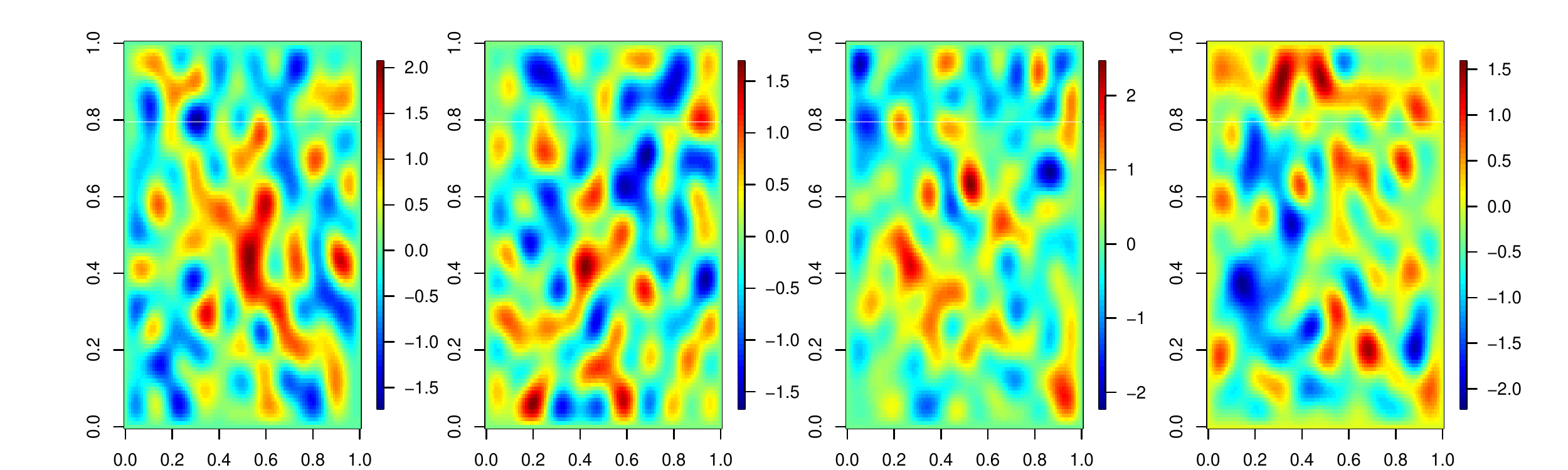}
  \caption{Heat maps of snapshots of a sample solution of 2D stochastic heat equation with mean reversion on a rectangular physical space domain taken at four time instants (from left to right: t=0.2, 0.4, 0.6 and 0.8).}
  \label{fig:sample_solution2D}
\end{figure}

\begin{figure}
	\includegraphics[scale=0.45, keepaspectratio]{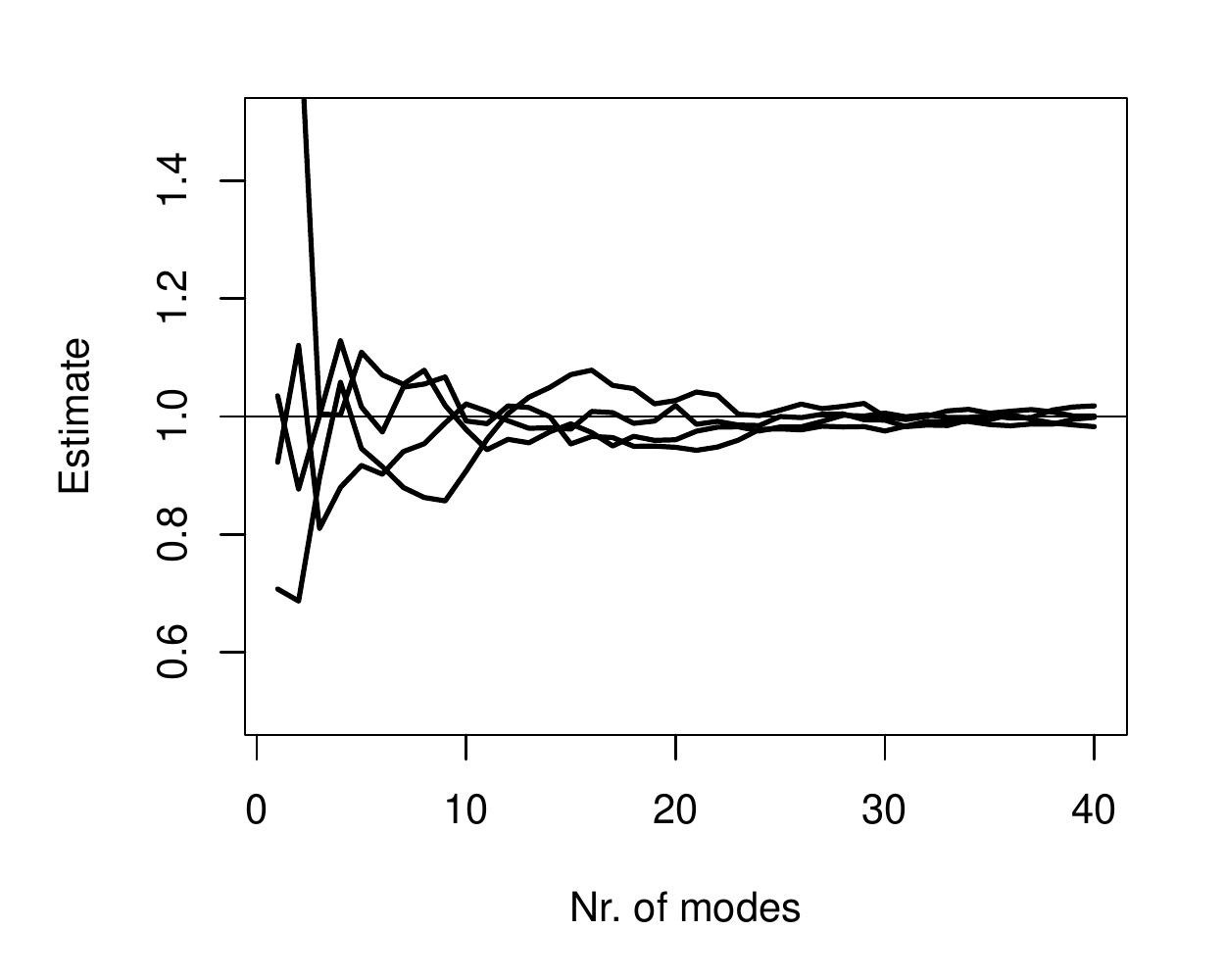}
	\includegraphics[scale=0.45, keepaspectratio]{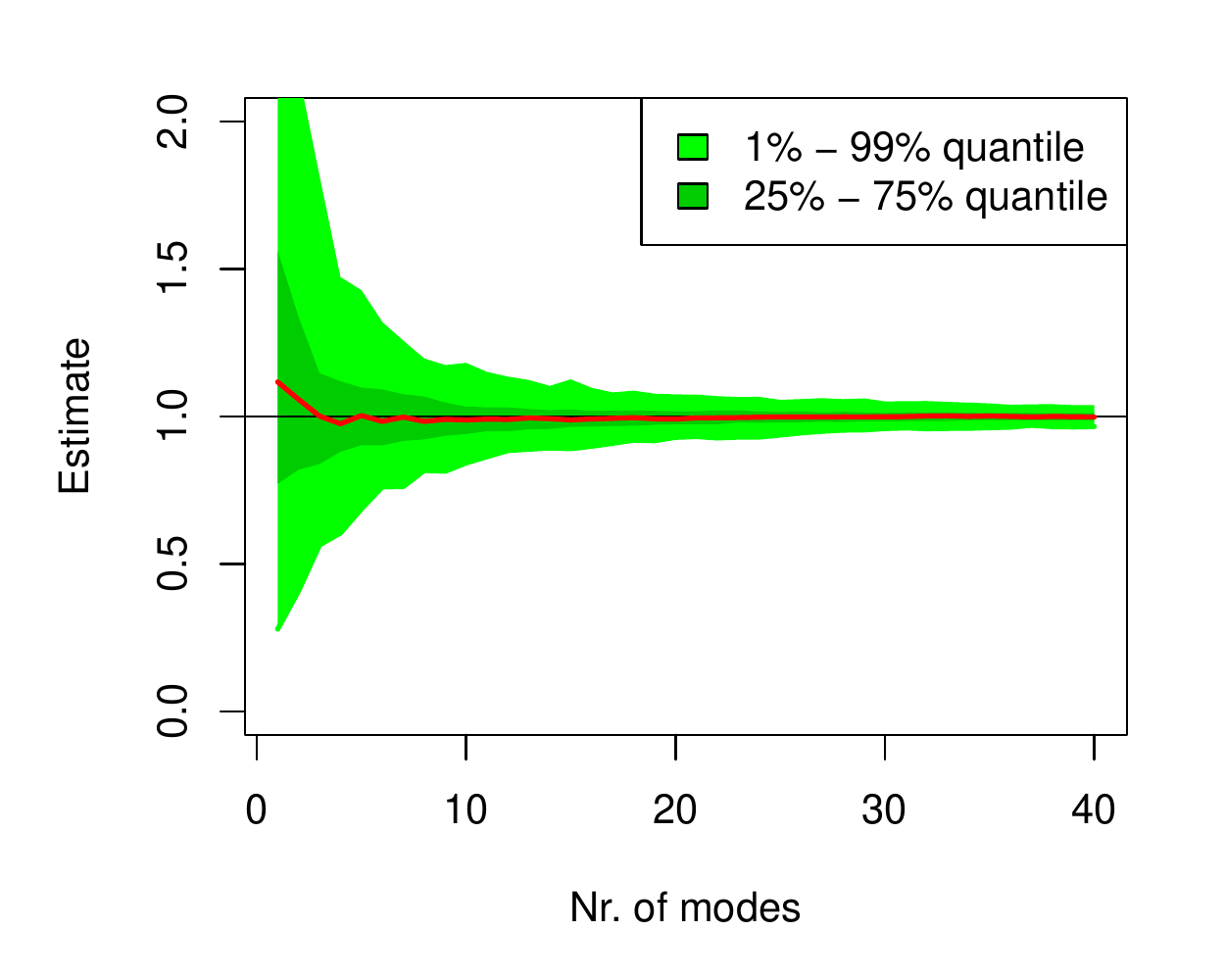}
	\includegraphics[scale=0.45, keepaspectratio]{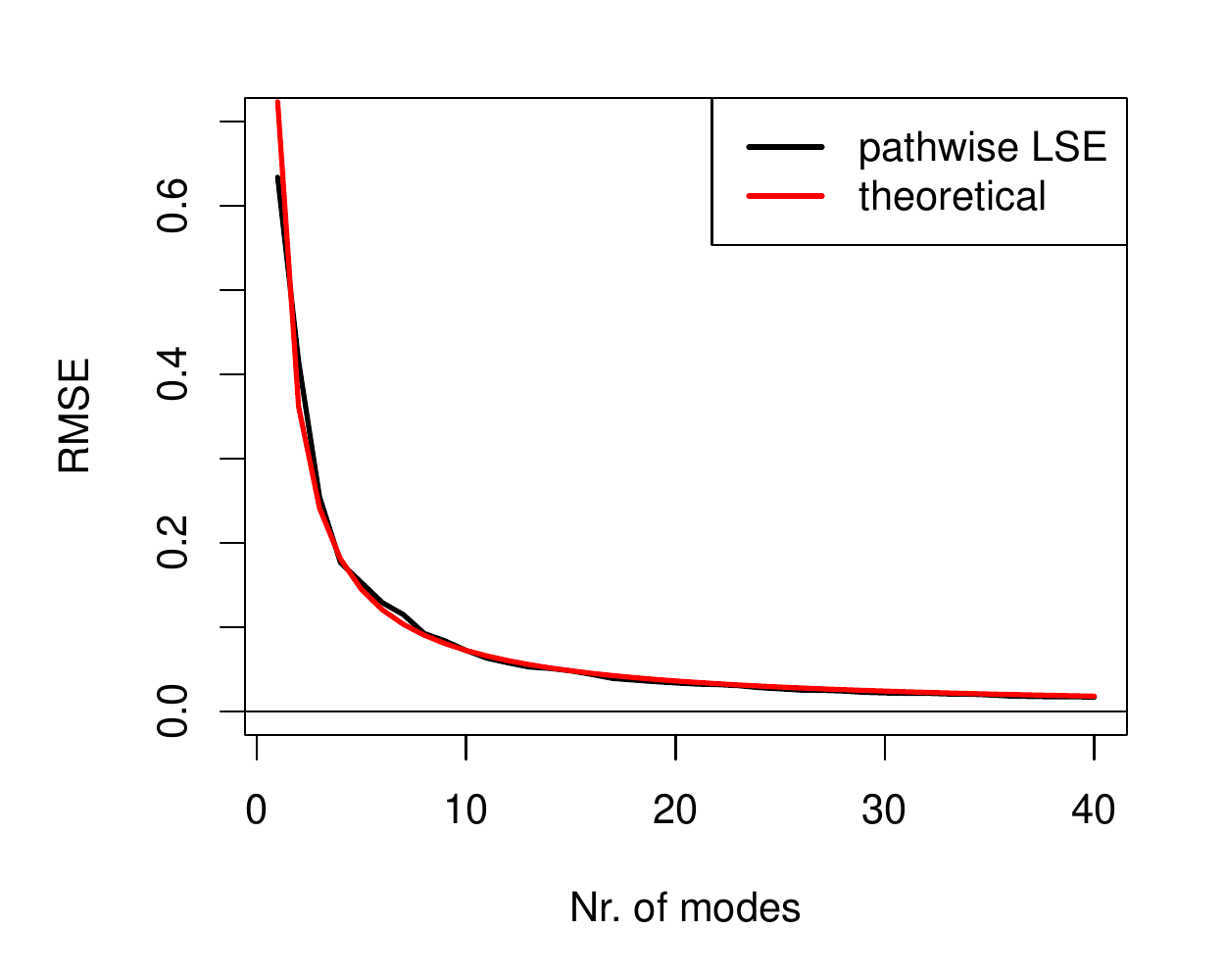}
	\includegraphics[scale=0.45, keepaspectratio]{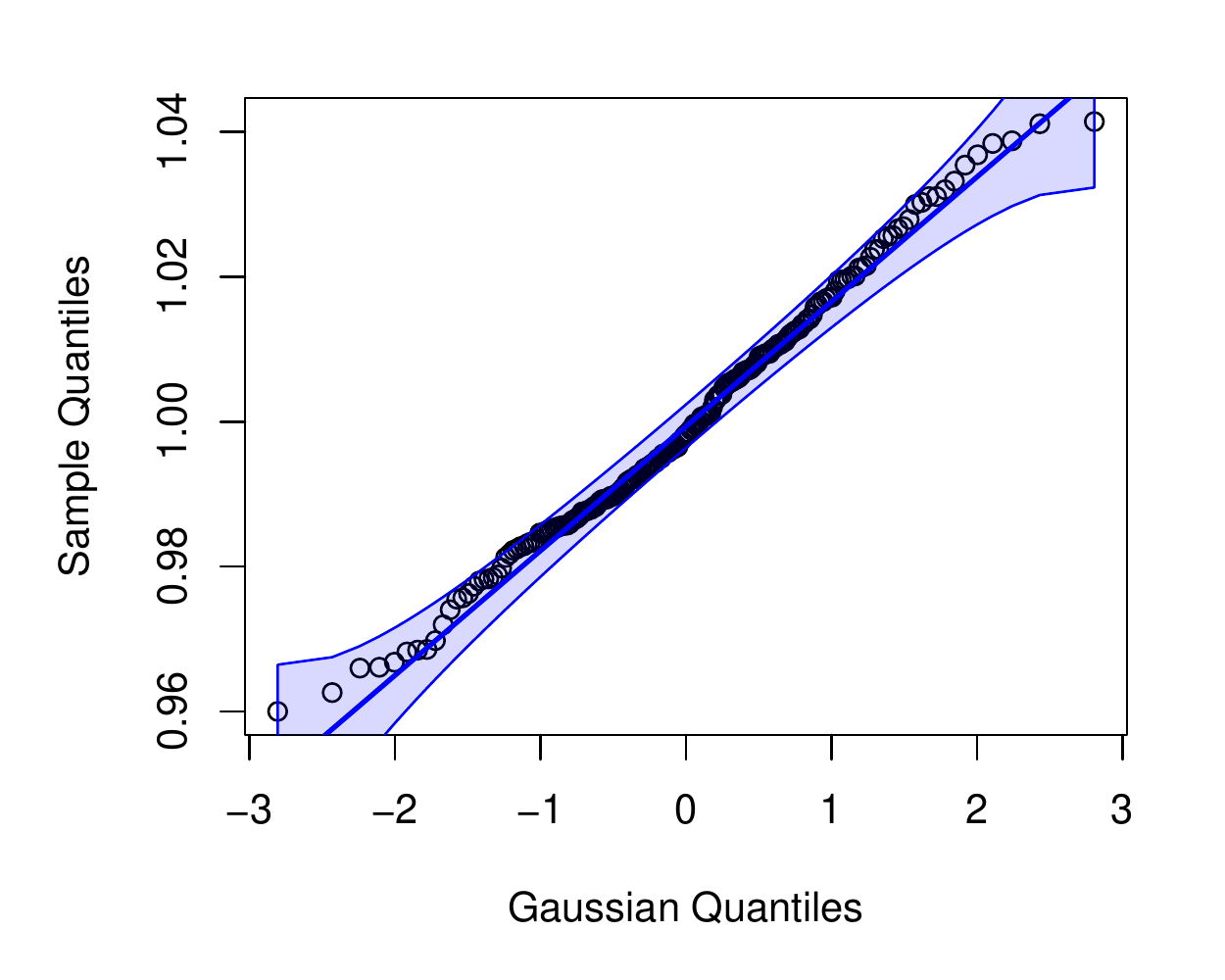}
  \caption{Characteristics of the convergence of pathwise LSE for 2D heat equation with mean reversion. 200 runs, Zero IC and H=0.3. \\
	Top-left: Four trajectories of estimates calculated from four different samples (realizations).\\
	Top-right: Sample quantile ranges (green areas) and median (red line), true value is 1.\\
	Bottom-left: Convergence of sample RMSE and theoretical RMSE, see \eqref{eq:speed HE lamba_1}.\\ 
	Bottom-right: Q-Q plot for sample of estimates (from 40 Fourier modes) including 95\% confidence envelope for normal distribution.
	}
  \label{fig:sim_results2D}
\end{figure}

\subsection{Discussion}
Results of the simulations confirm our findings on strong consistency of the pathwise LSE and on the speed of the convergence. In case of zero initial condition, it behaves similarly to the weighted MCE. However, for non-zero initial condition, pathwise LSE significantly outperforms the weighted MCE. This is in accordance with similar findings from finite-dimensional models (cf. \cite{KS20_1} or \cite{KS20_2}), where LSE-type estimators improve with far-from-stationary initial conditions (the drift part dominates the noise part  in the dynamics of the process in such case), whereas the ergodic-type estimators (such as the weighted MCE) reflecting long-term stationary behavior are ruined.

The theory for the weighted MCE, developed in \cite{Kriz2020}, covers only single-operator equations, so it is not directly applicable for the second example (2D stochastic heat equation with coupling with surroundings).  To our best knowledge, the pathwise LSE is the first efficient tool for parameter estimation in this equation (important in oceanography, for example) when the noise is generated by a cylindrical fractional Brownian motion with general Hurst parameter $0< H <1$. 

Although we did not address the asymptotic normality of the pathwise LSE in this paper (to keep it concise), we conjecture that it holds. Our conjecture is based on simulation results and asymptotic similarity with the weighted MCE, whose asymptotic normality has been proven in \cite{Kriz2020}.  

\section{Conclusions}\label{sect:Conclusions}
Using the least-squares formalism we have derived a spectral version of the least-squares estimator of the drift parameter in linear SPDEs with additive fractional noise for arbitrary $H \in (0,1)$. This estimator is strongly consistent in space, however, it can not be directly implemented pathwise, since it contains divergence-type stochastic integral w.r.t. a fractional process. This fact prevents its simulations and its use in practice. We addressed this issue by eliminating the stochastic integral getting a novel pathwise LSE. This estimator is constructed trully pathwise and, moreover, is robust to minor perturbations of the observed process (it does not contain differentiation of the observed trajectory), which makes it particularly useful in practice.  

Using the standard tools from classical probability and analysis, as well as modern tools from Malliavin calculus, we demonstrated that the newly developed pathwise LSE is strongly consistent in space and we conjecture that it is also asymptotically normal. Simulations reveal that it can efficiently utilize information about the unknown parameter from both non-stationary and stationary parts to the observed process.

The presented pathwise least-squares estimation procedure appears to be promising approach for various equations driven by fractional Brownian motion. The potentially interesting directions of further research may include (but are not limited to) proving asymptotic normality, studying different observation/sampling schemes, different asymptotic regimes, or generalization to semilinear stochastic equations driven by fractional Brownian motion.  


\appendix

\section{Elements from Malliavin calculus}

In the Appendix, we describe some basic constructions from Malliavin calculus in the setting of our problem. These constructions form a key ingredient for the study of asymptotic behavior of our estimators. 

\subsection{Construction of an infinite-dimensional isonormal Gaussian process}

Let $W^{(k)}=\big\{W^{(k)}(h^{(k)}),\, h^{(k)}\in\mathcal{H}^{(k)}\big\}$ be an isonormal Gaussian process generated by a fractional Brownian motion $\{B_k^H(t),t\geq 0\}$ for any $k\in\mathbb{N}$ and defined over a separable Hilbert space $\mathcal{H}^{(k)}$. The processes $W^{(k)}$ are mutually independent. 

Define
\[\mathcal{H}=\bigoplus_{n=1}^{\infty} \mathcal{H}^{(n)}\]
as a direct sum of $\mathcal{H}^{(n)}$'s, i.e.
\[\mathcal{H}=\big\{(h^{(1)},h^{(2)},\ldots),\ h^{(n)}\in\mathcal{H}^{(n)} \wedge \sum_{n=1}^{\infty}\|h^{(n)}\|^2_{\mathcal{H}^{(n)}}<\infty\big\},\]
with the scalar product
\[\big\langle(h^{(1)},h^{(2)},\ldots),(g^{(1)},g^{(2)},\ldots)\big\rangle_{\mathcal{H}}=\sum_{n=1}^{\infty}\langle h^{(n)},g^{(n)}\rangle_{\mathcal{H}^{(n)}}.\]
Then $\mathcal{H}$ is a separable Hilbert space.

For $h=(h^{(1)},h^{(2)},\ldots)\in\mathcal{H}$ define
\[W(h)=\lim_{N\rightarrow\infty}\sum_{n=1}^{N} W^{(n)}(h^{(n)}).\]
Note that this limit, understood in $L^2$-sense, is well-defined because 
\[\E\Big(\sum_{n=N}^{\infty} W^{(n)}(h^{(n)})\Big)^2=\sum_{n=N}^{\infty}\mathbb{E}\big(W^{(n)}(h^{(n)})\big)^2=\sum_{n=N}^{\infty}\|h^{(n)}\|^2_{\mathcal{H}^{(n)}}\xrightarrow[N\rightarrow\infty]{}0.\]
Since 
\begin{align*}
\E\big(W(h) W(g)\big)&=\E \Big(\sum_{n=1}^{\infty} W^{(n)}(h^{(n)})\Big)\Big(\sum_{n=1}^{\infty} W^{(n)}(g^{(n)})\Big)\\&=\sum_{n=1}^{\infty}\E\big(W^{(n)}(h^{(n)}) W^{(n)}(g^{(n)})\big)=\sum_{n=1}^{\infty}\langle h^{(n)},g^{(n)}\rangle_{\mathcal{H}^{(n)}}=\langle h,g\rangle_{\mathcal{H}}
\end{align*}
for any $h=(h^{(1)},h^{(2)},\ldots),g=(g^{(1)},g^{(2)},\ldots)\in\mathcal{H}$ the process $\big\{W(h),h\in\mathcal{H}\big\}$ is an isonormal Gaussian.

Denote by $C\!H_n^{(k)}$ the $n$-th Wiener chaos associated with $W^{(k)}$ for any $k\in\mathbb{N}$, i.e. $C\!H_n^{(k)}$ is the closed (in $L^2(\Omega)$) linear subspace generated by
\[ \big\{H_n\big(W^{(k)}(h^{(k)})\big),\|h^{(k)}\|_{\mathcal{H}^{(k)}}=1 \big\}, \]
where $H_n$ is the $n$-th Hermite polynomial. Similarly, let $C\!H_n$ be the $n$-th Wiener chaos associated with $W$, i.e. $C\!H_n$ is the closed linear subspace generated by 
\[\big\{H_n\big(W(h)\big),\|h\|_{\mathcal{H}}=1 \big\}. \] Fix $k\in\N$, take arbitrary $h^{(k)}\in\mathcal{H}^{(k)}$ such that $\|h^{(k)}\|_{\mathcal{H}^{(k)}}=1$ and define \[h=(0,\ldots,0,\underset{\quad \uparrow k\mathrm{th}}{h^{(k)}},0,\ldots,0).\]
Then
\[\|h\|^2_{\mathcal{H}}=\|h^{(k)}\|^2_{\mathcal{H}^{(k)}}=1\]
and
\[W(h)=\sum_{n=1}^{\infty} W^{(n)}(h^{(n)})=W^{(k)}(h^{(k)}).\]
Consider $X=H_n\big(W^{(k)}(h^{(k)})\big)\in C\!H_n^{(k)}$. Then  $X = H_n\big(W(h)\big)$ and $X\in C\!H_n$. By the linearity and $L^2(\Omega)$-closedness of $C\!H_n$ it follows that
\begin{equation}\label{eq: chaos embedding}
    C\!H_n^{(k)}\subset C\!H_n\quad \forall k\in\N.
\end{equation}

\subsection{Selected functionals of the solution in the second chaos}

It can be shown (see e.g. \cite{HNZ17}) that 
\[\int_0^T x_k(t)\d B^H_k(t) \in C\!H_2^{(k)}, \quad \int_0^T x_k^2(t)\d t - \E \left[\int_0^T x_k^2(t)\d t\right] \in C\!H_2^{(k)}. \]
Using the construction of the infinite-dimensional isonormal Gaussian process above and the chaos embedding \eqref{eq: chaos embedding}, we easily get for  
\[C_N= \sum_{k=1}^N \alpha_k\int_0^T x_k(t)\d B^H_k(t),\]
and
\[D_N= \sum_{k=1}^N \alpha_k^2\int_0^T x_k^2(t)\d t,\]
that
\begin{equation}\label{eq: C_N D_N in CH_2}
C_N \in C\!H_2\quad\text{ and }\quad D_N - \E D_N  \in C\!H_2 \quad \forall N\in \mathbb{N}.    
\end{equation}

\subsection{Almost sure convergence on a fixed chaos}
The following lemma is an important tool to demonstrate strong consistency of our estimators. It turns the $L^2$-convergence of a sequence of variables on a fixed chaos into the almost sure convergence. This idea, based on the combination of hypercontractivity with Borel--Cantelli 0-1 law, has already been used before, but we decided to formulate it as a separate lemma, because we believe it is interesting on its own.

\begin{Lemma}\label{Lem: a.s. convergence} 
Let $F_N,N\in\N,$ be random variables in a fixed Wiener chaos of an isonormal Gaussian process such that for some constants $c>0$ and $\delta>0$ the inequality 
\begin{equation}
\E F^2_N \leq \frac{c}{N^{\delta}}\nonumber
\end{equation}
holds for any $N\in\N$. Then $F_N\xrightarrow[N\rightarrow\infty]{} 0$ a.s.
\end{Lemma}

\begin{proof}
Consider parameters~$\zeta$ and~$\eta$ so that $0< \zeta < \delta/2$ and $\eta > \frac{1}{\delta/2 - \zeta}$. Denote by $C$ a~positive constant (independent of~$N$), which may change from line to line and calculate 
$$\P\big(|F_N|> N^{-\zeta}\big)\leq \frac{\E|F_N|^{\eta}}{N^{-\zeta\eta}}\leq C\frac{(\E F^2_N)^{\eta/2}}{N^{\zeta\eta}}\leq C \frac{1}{N^{\eta(\delta/2-\zeta)}},$$ where Chebyshev's inequality and hypercontractivity property on the fixed Wiener chaos (see e.g.~\cite{NP12}, Theorem 2.7.2) were used. Application of Borel--Cantelli lemma yields the almost sure convergence.

\end{proof}

\section{Cylindrical fractional Brownian motion and fractional Gaussian noise}\label{Appendix:cfBm}

The use of cylindrical fractional Brownian motion as a rigorous model for fractional Gaussian noise $\{\eta^H(t,\xi);\,t>0, \xi\in\mathcal O\}$ from equation \eqref{eq: HE_example} can be found in several publications. More detailed exposition can be found e.g. in \cite{DMPD2009}. In this Appendix, we briefly outline the connection between the two objects.\\

Let us begin with white noise in spatial dimension $\mathcal O \subset \mathbb{R}^d$ (no time coordinate considered). A natural space associated with \eqref{eq: HE_example} is $V=L^2(\mathcal O)$ with an orthonormal basis $\{e_k\}_{k=1}^{\infty}$. Take a sequence of i.i.d. real-valued random variables $\{U_k\}_{k=1}^{\infty}$ having standard Gaussian distribution. Now we can formally model the white noise in space as
\begin{equation}\label{WSN_def}
\eta(\xi) = \sum_{k=1}^{\infty} e_k(\xi) U_k, \quad \xi \in \mathcal{O}.
\end{equation}
Indeed, take arbitrary $f,g \in L^2(\mathcal O)$ and calculate (again formally)
\begin{equation}\nonumber
\int_{\mathcal{O}} f(\xi) \eta(\xi) \d \xi = \sum_{k=1}^{\infty} U_k \biggl(\int_{\mathcal{O}} f(\xi) e_k(\xi) \d \xi \biggr),  
\end{equation}
which is a centered Gaussian random variable with variance $||f||^2_V$, and
\begin{align*}
\E &\biggl(\int_{\mathcal{O}} f(\xi) \eta(\xi) \d \xi \biggr) \biggl(\int_{\mathcal{O}} g(\xi) \eta(\xi) \d \xi \biggr) = \E \biggl(\sum_{k=1}^{\infty} U_k \left\langle f,e_k\right\rangle_V \biggr) \biggl(\sum_{k=1}^{\infty} U_k \left\langle g,e_k\right\rangle_V \biggr) \\
&= \sum_{k=1}^{\infty} \left\langle f,e_k\right\rangle_V \left\langle g,e_k\right\rangle_V = \int_{\mathcal{O}} f(\xi) g(\xi) \d \xi.
\end{align*}
In particular, by taking $f$ and $g$ the indicator functions, we get that the noise is uncorrelated on arbitrary two disjoint subsets of $\mathcal{O}$ and it has constant intensity (variance is proportional to the volume of the subset). 
\begin{Remark}
Although the series in \eqref{WSN_def} does not converge in the $L^2(\Omega)$ sense in the space $V$, it does converge in some larger Hilbert space $\tilde{V}$ so that there exits a Hilbert-Schmidt embedding of $V$ into $\tilde{V}$ (this corresponds to the fact that white noise can not be considered as a proper function, but rather as a~distribution).
\end{Remark}

Continue with modeling of the fractional (autocorrelated) noise in time domain. Recall that real-valued fractional Brownian motion $\{b^{H}(t), t \geq 0\}$ with Hurst index $H \in (0,1)$ is a centered Gaussian process with continuous trajectories starting from zero and with autocovariance function 
\[
\E b^{H}(t) b^{H}(s) = \frac{1}{2}\bigl(t^{2H} + s^{2H} - |t-s|^{2H}\bigr).
\] 
As a consequence, it has stationary Gaussian increments. If $H=1/2$ we obtain standard Wiener process (with independent increments), if $H<1/2$ we have negatively correlated increments and for $H>1/2$ we have positively correlated increments. Formal derivative of fBm (which exists only in distributional sense) thus represents a good model for fractional noise in time.\\

To conclude, take the family of independent real-valued fractional Brownian motions $\{B^H_k(t), t \geq 0\}, k=1,2,\ldots$ and construct the cylindrical fBm on $V=L^2(\mathcal O)$ using an orthonormal basis $\{e_k\}_{k=1}^{\infty}$ as a formal sum 
\begin{equation}\nonumber 
B^H(t) = \sum_{k=1}^{\infty} e_k B^H_k(t).
\end{equation}
Increments (differentials) of this cylindrical fBm can serve as a model for the noise that is white in space and fractional in time. Indeed, for a fixed $0<s<t$ we have
\[
B^H(t) - B^H(s) = \sum_{k=1}^{\infty} e_k \bigl(B^H_k(t) - B^H_k(s)\bigr),
\]
which is the (scaled) white-in-space noise. 
Next, for a fixed measurable $D \subset \mathcal{O}$ with the finite volume take
\begin{align*}
\left\langle B^H(t), 1_D \right\rangle_V &= \sum_{k=1}^{\infty} \biggl(\int_{\mathcal O} 1_D(\xi) e_k(\xi) \d \xi \biggr) B^H_k(t) \\
&= \sum_{k=1}^{\infty} w_k B^H_k(t) = \beta^{H}(t),
\end{align*}
which itself is a (scaled) fractional Brownian motion, since $\sum_{k=1}^{\infty} w^2_k = \|1_D\|^2_V < \infty$ and so it represents an integrated fractional-in-time noise.

\section*{Acknowledgments}
We would like to thank to the anonymous referee for valuable comments and suggestions that helped to improve clarity and readability this paper.



\end{document}